\documentclass[11pt,reqno]{amsart}
\usepackage{graphicx}
\usepackage{color}
\usepackage{amsmath,amssymb,amsthm,amsfonts}
\usepackage{graphicx}
\usepackage{color}
\usepackage{multicol}
\def\R{\mathbb{R}}
\def\A{\mathcal{A}}
\makeatletter{\normalsize }

\newcommand{\Rmnum}[1]{\expandafter\@slowromancap\romannumeral #1@}
\makeatother

\newtheorem{thm}{Theorem}[section]

{\huge }

\newtheorem{lemma}[thm]{Lemma}
\newtheorem{remark}{Remark}[section]
\newtheorem{theorem}[thm]{Theorem}
\newtheorem{proposition}[thm]{Proposition}

\usepackage[numbers,sort&compress]{natbib}

\begin{document}
\author{Hai-Yang Jin}
\address{School of Mathematics, South China University of Technology, Guangzhou 510640, P.R. China}
\email{mahyjin@scut.edu.cn}
\author{Shijie Shi}
\address{College of Big Data and Internet, Shenzhen Technology University, Shenzhen 518118, P.R. China}
\email{shishijie@sztu.edu.cn}
\author{Zhi-an Wang}
\address{Department of Applied Mathematics, Hong Kong Polytechnic University,
Hung Hom, Kowloon, Hong Kong}
\email{mawza@polyu.edu.hk}

\title[Boundedness and Asymptotics of a reaction-diffusion system]{Boundedness and asymptotics of a reaction-diffusion system with density-dependent motility}

\begin{abstract}
We consider the initial-boundary value problem of a system of reaction-diffusion equations with density-dependent motility
\begin{equation*}\label{e1}\tag{$\ast$}
\begin{cases}
u_t=\Delta(\gamma(v)u)+\alpha u F(w) -\theta u, &x\in \Omega, ~~t>0,\\
v_t=D\Delta v+u-v,& x\in \Omega, ~~t>0,\\
w_t=\Delta w-uF(w),& x\in \Omega, ~~t>0,\\
 \frac{\partial u}{\partial \nu}=\frac{\partial v}{\partial \nu}= \frac{\partial w}{\partial \nu}=0,&x\in \partial\Omega, ~~t>0,\\
(u,v,w)(x,0)=(u_0,v_0,w_0)(x), & x\in\Omega,
\end{cases}
\end{equation*}
in a bounded domain $\Omega\subset\R^2$ with smooth boundary, $\alpha$ and $\theta$ are non-negative constants and $\nu$ denotes the outward normal vector of $\partial \Omega$. The random motility function $\gamma(v)$ and functional response function $F(w)$ satisfy the following assumptions:
\begin{itemize}
\item $\gamma(v)\in C^{3}([0,\infty)),~0<\gamma_{1}\leq\gamma(v)\leq \gamma_2,  \ |\gamma'(v)|\leq \eta$ for all $v\geq0$;
\item $F(w)\in C^1([0,\infty)), F(0)=0,F(w)>0 \ \mathrm{in}~(0,\infty)~\mathrm{and}~F'(w)>0 \ \mathrm{on}\ \ [0,\infty)$
\end{itemize}
for some positive constants $\gamma_1, \gamma_2$ and $\eta$.
Based on the method of weighted energy estimates and Moser iteration, we prove that the problem \eqref{e1} has a unique classical global solution uniformly bounded in time.  Furthermore we show that if $\theta>0$, the solution $(u,v,w)$ will converge to $(0,0,w_*)$ in $L^\infty$ with some $w_*>0$ as time tends to infinity, while if $\theta=0$, the solution $(u,v,w)$ will asymptotically converge to $(u_*,u_*,0)$ in $L^\infty$ with $u_*=\frac{1}{|\Omega|}(\|u_0\|_{L^1}+\alpha\|w_0\|_{L^1})$ if $D>0$ is suitably large.

\end{abstract}



\maketitle

\numberwithin{equation}{section}
\section{Introduction and main results}
 The reaction-diffusion models can generate a wide variety of exquisite spatio-temporal patterns arising in embryogenesis and development due to the diffusion-driven (Turing) instability \cite{Murray-book-2001, Kondo-Miura-Science-2010}. In addition, colonies of bacteria and eukaryotes can also generate rich and complex patterns driven by chemotaxis, which typically result from coordinated cell movement, growth and differentiation that often involve the detection and processing of extracellular signals \cite{Berg91, Berg95}. Many of these models invoke nonlinear diffusion which is enhanced by the local environment condition because of population pressure (cf. \cite{mendez2012density}), volume exclusion (cf. \cite{Dyson, Painter-Hillen-CAMQ}) or avoidance of danger (cf. \cite{Murray-book-2001}) and so on. By employing a synthetic biology approach, the authors of \cite{Liu-Science} introduced the so-called ``self-trapping'' mechanism into programmed bacterial {\it Eeshcrichia coli} cells which excrete signalling molecules acyl-homoserine lactone (AHL) such that at low AHL level, the bacteria undergo run-and-tumble random motion and are motile, while at high AHL levels, the bacteria tumble incessantly and become immotile due to the vanishing macroscopic motility. As a result, {\it Eeshcrichia coli} cells formed the outward expanding stripe (wave) patterns in the petri dish. To gain a quantitative understanding of the patterning process in the experiment, the following three-component reaction-diffusion system has been proposed in \cite{Liu-Science}:
\begin{equation}\label{OR-1}
\begin{cases}
u_t=\Delta(\gamma(v)u)+ \frac{\alpha w^2 u}{w^2+\lambda}, &x\in \Omega, ~~t>0,\\
v_t=D\Delta v+u-v,& x\in \Omega, ~~t>0,\\
w_t=\Delta w-\frac{ w^2 u}{w^2+\lambda},& x\in \Omega, ~~t>0,\\
\end{cases}
\end{equation}
where $u(x,t),v(x,t),w(x,t)$ denote the bacterial cell density, concentration of acyl-homoserine lactone (AHL) and nutrient density, respectively; $\alpha,\lambda,D>0$ are constants and $\Omega$ is  bounded domain in $\R^n (n\geq 2)$. The first equation of \eqref{OR-1} describes the random motion of bacterial cells with an AHL-dependent motility coefficient $\gamma(v)$, and a cell growth due to the nutrient intake. The second equation of \eqref{OR-1} describes the  diffusion, production and turnover of AHL, while the third equation provides the dynamics of diffusion and consumption for the nutrient. The prominent feature of the system \eqref{OR-1} is that the cell diffusion rate depends on a motility function $\gamma(v)$ satisfying $\gamma'(v)<0$, which takes into account the repressive effect of AHL concentration on the cell motility (cf. \cite{Liu-Science}).

Though the system \eqref{OR-1} may numerically reproduce some key features of experimental observations as illustrated in \cite{Liu-Science}, the mathematical analysis remains open. Later an alternative simplified two-component so-called ``density-suppressed motility'' model was proposed in \cite{Fu-PRL}:
\begin{equation}\label{OR-2}
\begin{cases}
u_t=\Delta(\gamma(v)u)+\mu u(1-u),&x\in\Omega,t>0,\\
v_t=D\Delta v+u-v,& x\in \Omega, ~~t>0,
\end{cases}
\end{equation}
where the reduced growth rate of cells at high density was used to approximate the nutrient depletion effect in the system \eqref{OR-1}. One can expand the Laplacian term in the first equation of \eqref{OR-2} to obtain a chemotaxis model with signal-dependent motility. Hence the system \eqref{OR-2}  shares some features  similar to the Keller-Segel type chemotaxis model. However  due to the cross-diffusion and the density-suppressed motility (i.e., $\gamma'(v)<0$), even for the simplified system \eqref{OR-2}, there are only few results obtained recently when the Neumann boundary conditions are imposed, as summarized below.
\begin{itemize}
\item[(1)]$\mu>0$: In this case, the first result on the global existence and large time behavior of solutions was established in \cite{JKW-SIAP-2018}. More precisely, it is shown in \cite{JKW-SIAP-2018} that the system \eqref{OR-2} has a unique global classical solution in two dimensional spaces for the motility function $\gamma(v)$ satisfying the assumptions:
$\gamma(v)\in C^3([0,\infty)), \gamma(v)>0\ \  \mathrm{and}~\gamma'(v)<0 \ \ \mathrm{on} ~[0,\infty)$, $\lim\limits_{v \to \infty}\gamma(v)=0$ and $\lim\limits_{v \to \infty}\frac{\gamma'(v)}{\gamma(v)}$ exists. Moreover, the constant steady state $(1,1)$ of \eqref{OR-2} is proved to be globally asymptotically stable if
$
\mu>\frac{K_0}{16}$ where $ K_0=\max\limits_{0\leq v \leq \infty}\frac{|\gamma'(v)|^2}{\gamma(v)}$. Recently, the global existence result has been extended to the higher dimensions ($n\geq 3$)  for large $\mu>0$ in \cite{WW-JMP-2019}. On the other hand, for small $\mu>0$, the existence/{\color{black}nonexistence} of nonconstant steady states  of \eqref{OR-2} was rigorously established under some constraints on the parameters in \cite{MPW-PD-2019} and the periodic pulsating wave is analytically obtained by the multi-scale analysis. When $\gamma(v)$ is a constant step-wise function, the dynamics of discontinuity interface was studied in \cite{SIK-EJAP-2019}.
\item[(2)]$\mu=0$: The existence of global  classical solutions of \eqref{OR-2} in any dimensions has been established in \cite{YK-AAM-2017} in the case of  $\gamma(v)=c_0/v^{k}(k>0)$ for small $c_0>0$. The smallness assumption on $c_0$ is removed lately for the parabolic-elliptic case with $0<k<\frac{2}{n-2}$ in \cite{AY-Nonlinearity-2019}.
Moreover, the global classical solution in two dimensions and global weak solution in three dimensions of \eqref{OR-2} with $\mu=0$ are obtained in \cite{TW-M3AS-2017} under the following assumptions:
\vspace{0.2cm}

\begin{quote}
\begin{itemize}
\item[(H1)] $\gamma(v)\in C^{3}([0,\infty)),$ and there exist $\gamma_{1},\gamma_{2},\eta>0$ such that $0<\gamma_{1}\leq\gamma(v)\leq \gamma_2$, $|\gamma'(v)|\leq \eta$ for all $v\geq0$.
\end{itemize}
\end{quote}
\vspace{0.2cm}

Without the lower-upper bound hypotheses for $\gamma(v)$ as assumed in (H1), {\color{black} if $\gamma(v)$ decays algebraically and $1\leq n\leq 3$, the global existence of weak solutions with large initial data was established in \cite{Kim-NARWA-2019}}. Moreover, if $\gamma(v)$ decays to zero fastly like exponential decay, the solution of \eqref{OR-2} with $\mu=0$ may blow up. For example, if $\gamma(v)=e^{-\chi v}$, by constructing a Lyapunov functional, it is proved in \cite{JW-PAMS-2019} that  there exists a critical mass $m_*=\frac{4\pi}{\chi}$ such that the solution of \eqref{OR-2} with $\mu=0$ exists globally with uniform-in-time bound if $\int_\Omega u_0dx<m_*$ while blows up if $\int_\Omega u_0dx>m_*$ in two dimensions, where $u_0$ denotes the initial value of $u$.
\end{itemize}

Except the above mentioned results on the simplified model \eqref{OR-2},
to our knowledge, there are not any results available for the original three-component system \eqref{OR-1} proposed in \cite{Liu-Science}. The purpose of this paper is to develop some analytical results on the system \eqref{OR-1}. More generally we shall consider the following initial-boundary value problem
\begin{equation}\label{1-1}
\begin{cases}
u_t=\Delta(\gamma(v)u)+\alpha u F(w) -\theta u, &x\in \Omega, ~~t>0,\\
v_t=D\Delta v+u-v,& x\in \Omega, ~~t>0,\\
w_t=\Delta w-uF(w),& x\in \Omega, ~~t>0,\\
 \frac{\partial u}{\partial \nu}=\frac{\partial v}{\partial \nu}= \frac{\partial w}{\partial \nu}=0,&x\in \partial\Omega, ~~t>0,\\
(u,v,w)(x,0)=(u_0,v_0,w_0)(x), & x\in\Omega,
\end{cases}
\end{equation}
where $\theta\geq 0$ accounts for the natural death rate. We assume that the motility function $\gamma(v)$ satisfies the {\color{black}assumption} (H1) as used in  \cite{TW-M3AS-2017} and the intake rate function $F(w)$  satisfies the following conditions
\vspace{0.2cm}

\begin{itemize}
\item[(H2)] $F(w)\in C^1([0,\infty)), F(0)=0,F(w)>0 \ \mathrm{in}~(0,\infty)~\mathrm{and}~F'(w)>0 \ \mathrm{on}\ \ [0,\infty).$
\end{itemize}
\vspace{0.2cm}

The conditions in (H2) can be satisfied by a wide class of functions such as
\begin{equation*}
F(w)=w,\ \ \ \ F(w)=\frac{w}{\lambda+w},\ \ \ F(w)=\frac{w^m}{\lambda+w^m},
\end{equation*}
with constants $\lambda>0$ and $m>1$, which are called the Holling type functional response functions  in the predator-prey system (cf. \cite{JW-JDE-2017,JW-EJAP-2019,WSW-JDE-2016,WWS-M3AS-2018}). Therefore the system \eqref{OR-1} is a special case of the equations in \eqref{1-1} with $\theta=0$ and $F(w)=\frac{w^2}{\lambda+w^2}$. In the sequel, for brevity we shall drop the differential element in the integrals without confusion, namely abbreviating $\int_\Omega fdx$ as $\int_\Omega f$ and $\int_0^t\int_\Omega fdxd\tau$ as $\int_0^t\int_\Omega f$. With the assumptions (H1)-(H2), we first prove the existence of globally bounded solutions to the system \eqref{1-1} in two dimensions as follows.
\begin{theorem}[Global boundedness]\label{GB}
Let $\Omega\subset\R^2$ be a bounded domain with smooth boundary, and the assumptions (H1)-(H2) hold. Assume $(u_0,v_0,w_0)\in [W^{1,\infty}(\Omega)]^3$ with $u_0, v_0,w_0\gneqq 0$. Then for any $\theta\geq 0$,  the problem \eqref{1-1} has a unique global classical solution $(u,v,w) \in [C^0(\bar{\Omega}\times[0,\infty))\cap C^{2,1}(\bar{\Omega}\times(0,\infty))]^3$ satisfying $u,v,w\gneqq 0$ for all $t>0$ and
\begin{equation*}
\|u(\cdot,t)\|_{L^\infty}\leq M,
\end{equation*}
where $M>0$ is s constant such that
\begin{equation}\label{M}
{M:=C_{1}(1+\alpha)^{13}(1+\frac{1}{D})^{12}e^{C_{2}(1+\alpha)^{6}(1+\frac{1}{D})^4}},
\end{equation}
with some constants $C_1,C_2>0$ independent of $D,\alpha$ and $t$.
\end{theorem}
We remark that we precise the dependence of constant $M$ on {\color{black}$\alpha$} in \eqref{M} so that the results of Theorem \ref{GB} can be applied to the case $\alpha=0$. The explicit dependence of $M$ on $D$ will {\color{black}be used later} to derive the asymptotic stability of solutions when imposing  some {\color{black}conditions} on $D$ as shown in the next theorem.

\begin{theorem}[Asymptotic stability of solutions]\label{LTB}
	Let the assumptions in Theorem \ref{GB} hold and $(u,v,w)$ be the classical solution of \eqref{1-1} obtained in Theorem \ref{GB}. Then the following asymptotic stability results hold.
	\begin{itemize}
		\item[(1)] If $\theta>0$, then it holds that
\begin{equation*}
\lim\limits_{t\to\infty}(\|u(\cdot,t)\|_{L^\infty}+\|v(\cdot,t)\|_{L^\infty}+\|w(\cdot,t)-w_*\|_{L^\infty})=0,
\end{equation*}
where $w_{*}> 0$ is a constant determined by {\color{black}$w_*=\frac{1}{|\Omega|}\|w_0\|_{L^1}-\frac{1}{|\Omega|}\int_0^\infty\int_\Omega uF(w)$.}
\item[(2)] If $\theta=0$, there exists a constant $D_0>0$ such that if $D\geq D_0$, then
\begin{equation*}
	\lim\limits_{t\to\infty}(\|u(\cdot,t)-u_*\|_{L^\infty}+\|v(\cdot,t)-u_*\|_{L^\infty}+\|w(\cdot,t)\|_{L^\infty})=0,
\end{equation*}
where $u_*=\frac{1}{|\Omega|}(\|u_0\|_{L^1}+\alpha\|w_0\|_{L^1})$.
\end{itemize}	
\end{theorem}

\begin{remark} The results of Theorem \ref{LTB} hold for any $\alpha\geq0$. In the case $\theta=0$ and $\alpha=0$, the system \eqref{1-1} reduces to
\begin{equation}\label{1-1n}
\begin{cases}
u_t=\Delta(\gamma(v)u), &x\in \Omega, ~~t>0,\\
v_t=D\Delta v+u-v,& x\in \Omega, ~~t>0,\\
 \frac{\partial u}{\partial \nu}=\frac{\partial v}{\partial \nu}=0,&x\in \partial\Omega, ~~t>0,\\
(u,v)(x,0)=(u_0,v_0)(x), & x\in\Omega.
\end{cases}
\end{equation}
The global existence of classical solutions of \eqref{1-1n} in two dimensions has been established in \cite{TW-M3AS-2017}, whereas the large time behavior of solution is left open. The result of Theorem \ref{LTB}(2)  solves this open question for large $D>0$.
\end{remark}

\noindent{\bf Sketch the proof.} With the special structure of the first equation of \eqref{1-1}, we shall use some ideas in  \cite{JKW-SIAP-2018, TW-M3AS-2017} to show the boundedness of solutions. More precisely, let $\mathcal{A}$ be a self-adjoint realization of $-\Delta$ ({\color{black}see more details in \cite{SA}}) defined on $D(\mathcal{A}):=\{\phi\in W^{2,2}(\Omega)\cap L^{2}(\Omega)|\int_{\Omega}\phi=0\text{~~and~~}\frac{\partial\phi}{\partial\nu}=0
~~\text{on}~~\partial\Omega\}$. Let $\mathcal{B}$ denote the self-adjoint realization of $-\Delta+\delta$ under homogeneous Neumann boundary conditions in $L^2(\Omega)$ for some $\delta>0$. We can use the first and third equations of \eqref{1-1} to obtain
\begin{equation*}\label{0-2}
(u+\alpha w-\bar{u}-\alpha \bar{w})_t+\mathcal{A} (\gamma(v)u+\alpha w-\overline{\gamma(v)u}-\alpha\bar{w})=0, \ \text{if} \ \theta=0,
\end{equation*}
and
\begin{equation*}\label{A3}
		\begin{split}
		(u+\alpha w)_t+\mathcal{B}(\gamma(v)u+\alpha w)
		=(\delta \gamma(v)-\theta) u+\delta\alpha w, \ \text{if} \ \theta>0,
		\end{split}
		\end{equation*}
which enable us to find a constant $c_1>0$ independent of $D$ and $\alpha$ such that $\int_t^{t+\tau}\int_\Omega u^2\leq c_1 (1+\alpha)^2$ for some appropriately small $\tau\in(0,1]$.
Using the smoothing properties of the second equation of \eqref{1-1} we can obtain  the boundedness of $\int_\Omega |\nabla v|^2$ and $\int_t^{t+\tau}\int_\Omega |\Delta v|^2 $. Then we use the direct $L^2$ estimate of $u$ as developed in \cite{JKW-SIAP-2018} to find two positive  constants $c_2,c_3$ independent of $D$ and $\alpha$ such that
\begin{equation*}
		\|u(\cdot,t)\|_{L^2}\leq c_2(1+\alpha)e^{c_3(1+\alpha)^{6}(1+\frac{1}{D})^{4}} \ \ \ \mathrm{for\ all}\ \ t\in(0,T_{max}),
		\end{equation*}
see Lemma \ref{L2-L} for details. Then using the routine bootstrap argument and Moser-iteration method, we derive that $\|u(\cdot,t)\|_{L^\infty}\leq M$ with $M$ satisfying \eqref{M}.

To study the asymptotic behavior, we divide our proofs into two cases: $\theta>0$ and $\theta=0$. When $\theta>0$,  we can obtain  from the first equation of \eqref{1-1} that  $\int_0^\infty\int_{\Omega} u<\infty$,  which combined with  the relative compactness of $(u(\cdot,t))_{t>1}$ in $C(\Omega)$ (see Lemma $\ref{A10*}$) gives
$\|u(\cdot,t)\|_{L^\infty}\to 0$  and hence $\|v(\cdot,t)\|_{L^\infty}\to 0$ as $t\to \infty$ from the second equation of $\eqref{1-1}$. Then using the semigroup estimates and the decay property of $u$, from the third equation we can show that  {\color{black}$\|w(\cdot,t)-w_*\|_{L^\infty}\to 0$} as $t\to\infty$ for some $w_*>0$, where $w_*>0$ is proved by showing
	\begin{equation*}
		\int_\Omega \ln w(x,t)\geq -c_4, \ \mathrm{for\ all}\ t\geq 1
	\end{equation*}
for some constant $c_4>0$, see Lemma \ref{Lw*} for details.

  When $\theta=0$, from the third equation of \eqref{1-1} we have
  \begin{equation*}\label{Ac-2}
\int_0^\infty\int_\Omega uF(w)+\int_0^\infty \int_\Omega|\nabla w|^2<\infty,
\end{equation*}
which, combined with  $\|u_0\|_{L^1}\leq\|u(\cdot,t)\|_{L^1}$, entails us that
\begin{equation*}
\|w(\cdot,t)\|_{L^\infty}\to 0\ \  \mathrm{as} \ \ t\to\infty.
\end{equation*}
On the other hand, using the relations between $M$ and $D$ in \eqref{M}, we can construct a $L^2$ energy functional of $(u,v)$ which along with suitable regularity of $(u,v)$ finally leads to the convergence of $(u,v)$ as claimed.

\section{Local existence and Preliminaries}
{\color{black}The existence and uniqueness of local solutions of $\eqref{1-1}$ can be} readily proved by the Amann's theorem \cite{A-DIE-1990, A-Book-1993} (cf. also \cite[Lemma 2.6]{Wang-Hillen}) or the fixed point theorem along with the parabolic regularity theory \cite{TW-2015-SIMA,JKW-SIAP-2018}. We omit the details of the proof for brevity.

\begin{lemma}[Local existence]\label{LS}
	Let $\Omega \subset \R^2$ be a bounded domain with smooth boundary and the assumptions (H1) and (H2) hold. Assume $(u_0,v_0,w_{0})\in [W^{1,\infty}(\Omega)]^3$ with $u_0, v_0, w_{0}\gneqq 0 $. Then there exists $T_{max}\in (0,\infty]$ such that the problem $\eqref{1-1}$ has a unique classical solution $(u,v,w) \in [C^0(\bar{\Omega}\times[0,T_{max}))\cap C^{2,1}(\bar{\Omega}\times(0,T_{max}))]^3$ satisfying $u, v,w>0 $ for all $t>0$. Moreover,
\begin{equation*}\label{a-priori}
	if~~T_{max}<\infty,~~then~~\|u(\cdot,t)\|_{L^\infty}\to\infty ~~as~~t\nearrow T_{max}.
\end{equation*}
\end{lemma}

\begin{lemma}\label{UW}
	The solution $(u,v,w)$ of \eqref{1-1} satisfies
\begin{equation}\label{L1u}	\|u(\cdot,t)\|_{L^{1}}+\alpha\|w(\cdot,t)\|_{L^{1}}+\theta\int_{0}^{t}\|u(\cdot,s)\|_{L^{1}}=\|u_{0}\|_{L^{1}}+\alpha\|w_{0}\|_{L^{1}}, \ \ \ t\in(0,T_{max}),
\end{equation}
	and
\begin{equation}\label{Loow}
	\|w(\cdot,t)\|_{L^\infty}~\text{is decreasing in }t.
\end{equation}
	Moreover, for all $(x,t)\in\Omega\times(0,T_{\max}),$ it follows that
\begin{equation}\label{Fw}
	F(w(x,t))\leq C_{F}=F(\|w_{0}\|_{L^{\infty}}).
\end{equation}
\end{lemma}

\begin{proof}
	We first multiply the third equation of \eqref{1-1} by $\alpha$ and add the resulting equation to the first equation of \eqref{1-1}. Then integrating the result over $\Omega\times(0,t)$, we have \eqref{L1u} directly. The application of the maximum principle to the third equation of \eqref{1-1} gives \eqref{Loow}. Furthermore, since $F'(w)>0$ for all $w\geq 0$, one has \eqref{Fw} by using \eqref{Loow}.
\end{proof}

Next, we list some well-known estimates for the Neumann heat semigroup for later use.
\begin{lemma}[\cite{Winkler-2010-JDE}]\label{SE}
Let $(e^{t\Delta})_{t\geq 0}$ be the Neumann heat semigroup in $\Omega$, and let $\lambda_1>0$ denote the first nonzero eigenvalue of $-\Delta$ in $\Omega$ under Neumann boundary conditions. Then for all $t>0$, there exist some constants $k_i(i=1,2,3)$ depending only on $\Omega$ such that

$(\mathrm{i})$ ~~If $1\leq q\leq p\leq \infty$, then
\begin{equation}\label{Lp-1}
\|e^{t\Delta}z\|_{L^p}\leq k_{1}\left(1+t^{-\frac{n}{2}(\frac{1}{q}-\frac{1}{p})}\right)e^{-\lambda_1 t}\|z\|_{L^q}
\end{equation}
for all $z\in L^q(\Omega)$ satisfying $\int_{\Omega} z=0$.

$(\mathrm{ii})$ ~~ If $1\leq q\leq p\leq\infty$, then
\begin{equation}\label{Lp-2}
\|\nabla e^{t\Delta}z\|_{L^p}\leq k_2\left(1+t^{-\frac{1}{2}-\frac{n}{2}(\frac{1}{q}-\frac{1}{p})}\right)e^{-\lambda_1 t}\|z\|_{L^q}
\end{equation}
for all $z\in L^q(\Omega)$.

$(\mathrm{iii})$ ~~If $2\leq q\leq p<\infty$, then
\begin{equation}\label{Lp-3}
\|\nabla e^{t\Delta}z\|_{L^p}\leq k_{3}\left(1+t^{-\frac{n}{2}(\frac{1}{q}-\frac{1}{p})}\right)e^{-\lambda_1t}\|\nabla z\|_{L^q}
\end{equation}
for all $z\in W^{1,p}(\Omega)$.
\end{lemma}

The following lemma will be used to show the boundedness of solution, one can see \cite[Lemma 3.3]{LW-CPDE-2015} or \cite[Lemma 3.4]{SSW14} for details.
\begin{lemma}\label{G1}
	Let $T>0$, $\tau\in(0,T)$, $a>0$ and $b>0$. Suppose that $y:[0,T)\rightarrow[0,\infty)$ is absolutely continuous and fulfils
\begin{equation*}
	y'(t)+ay(t)\leq h(t),\quad\text{for all }t\in(0,T),
\end{equation*}
	with some nonnegative function $h\in L_{loc}^{1}([0,T))$ satisfying
\begin{equation*}
	\int_{t}^{t+\tau}h(s)ds\leq b,\quad\text{for all }t\in[0,T-\tau).
\end{equation*}
	Then
\begin{equation*}
	y(t)\leq\max\left\{y(0)+b,\frac{b}{a\tau}+2b \right\},\quad\text{for all }t\in(0,T).
\end{equation*}
\end{lemma}

\section{Boundedness of solutions (Proof of Theorem \ref{GB})}
In this section, we shall establish the boundedness of solution in two dimensions. 	
\begin{lemma}\label{L2u}
	Suppose the assumptions in Theorem \ref{GB} hold. For all $\theta\geq0$, there exists a constant $K_1>0$ independent of $D$ and $\alpha$ such that the solution to \eqref{1-1} satisfies
	\begin{equation}\label{2L1-2}
	\int_t^{t+\tau} \int_{\Omega} u^2\leq K_1(1+\alpha)^{2}, \ \ \mathrm{for\ all}\ \ t\in(0,\widetilde{T}_{max}),
		\end{equation}
		where
		\begin{equation*}\label{DtT}
		\tau:=\min \Big\{ 1,\frac{1}{2} T_{max}\Big\} \ \ \mathrm{and}\ \
		\widetilde{T}_{max}:=T_{max}-\tau.
		\end{equation*}
	\end{lemma}
	\begin{proof} We divide the proof into two cases: $\theta=0$ and $\theta>0$.

	\noindent {\bf{ Case 1: $\theta=0$.}} In this case, multiplying the third equation of \eqref{1-1} by $\alpha$ and adding the result to the first equation of \eqref{1-1}, one has
	\begin{equation}\label{0-1}
		(u+\alpha w)_t=\Delta (\gamma(v)u+\alpha w).
	\end{equation}
	Then integrating \eqref{0-1} with respect to $x$ with the homogeneous Neumann boundary conditions{\color{black}, one} has
	\begin{equation}\label{0-1*}
		\bar{u}+\alpha \bar{w}=\frac{1}{|\Omega|}\int_\Omega u_0+\alpha \frac{1}{|\Omega|}\int_{\Omega} w_0=\bar{u}_0+\alpha \bar{w}_0,
	\end{equation}
	where $\bar{f}$ denotes the mean of $f$, namely $\bar{f}=\frac{1}{|\Omega|}\int_\Omega fdx$. Let $\mathcal{A}$ be a self-adjoint realization of $-\Delta$ defined on $D(\mathcal{A}):=\{\phi\in W^{2,2}(\Omega)\cap L^{2}(\Omega)|\int_{\Omega}\phi=0\text{~~and~~}\frac{\partial\phi}{\partial\nu}=0~~\text{on}~~\partial\Omega\}$.
Then using \eqref{0-1*}, we can rewrite \eqref{0-1} as
\begin{equation}\label{0-2}
(u+\alpha w-\bar{u}-\alpha \bar{w})_t=-\mathcal{A} (\gamma(v)u+\alpha w-\overline{\gamma(v)u}-\alpha\bar{w}).
\end{equation}
Multiplying \eqref{0-2} by $\A^{-1}\left(u+\alpha w-\bar{u}-\alpha\bar{w}\right)$ and integrating the result by parts, we obtain
		\begin{equation*}\label{0-3}
		\begin{split}
		&\frac{1}{2}\frac{d}{dt}\int_{\Omega}|\A^{-\frac{1}{2}}(u+\alpha w-\bar{u}-\alpha\bar{w})|^{2}\\
		&=-\int_{\Omega}\A\left(\gamma(v)u+\alpha w-\overline{\gamma(v)u}-\alpha\bar{w}\right)\cdot\A^{-1}\left(u+\alpha w-\bar{u}-\alpha\bar{w}\right)\\
		&=-\int_{\Omega}(u+\alpha w-\bar{u}-\alpha\bar{w})\cdot\left(\gamma(v)u+\alpha w-\overline{\gamma(v)u}-\alpha\bar{w}\right)\\
&=-\int_{\Omega}\gamma(v)(u-\bar{u})^{2}-\alpha^{2}\int_{\Omega}(w-\bar{w})^{2}
-\alpha\int_{\Omega}\Big(1+\gamma(v)\Big)(u-\bar{u})(w-\bar{w})\\
&\ \ \ \ -\bar{u}\int_{\Omega}\gamma(v)(u-\bar{u})-\alpha\overline{u}\int_{\Omega}\gamma(v)(w-\bar{w}),\\
		\end{split}
		\end{equation*}
which together with the facts $0<\gamma_1\leq \gamma(v)\leq \gamma_{2}$ and the {nonnegativity} of $u,w$, gives
\begin{equation}\label{0-4}
\begin{split}
&\frac{1}{2}\frac{d}{dt}\int_{\Omega}|\A^{-\frac{1}{2}}(u+\alpha w-\bar{u}-\alpha\bar{w})|^{2}+\gamma_1\int_\Omega(u-\bar{u})^{2}
+\alpha^{2}\int_{\Omega}(w-\bar{w})^{2}\\
&=-\alpha\int_{\Omega}\Big(1+\gamma(v)\Big)(u-\bar{u})(w-\bar{w})-\bar{u}\int_{\Omega}\gamma(v)(u-\bar{u})
-\alpha\overline{u}\int_{\Omega}\gamma(v)(w-\bar{w})\\
&\leq \alpha \bar{w}\int_\Omega \Big(1+\gamma(v)\Big)u+ \alpha \bar{u}\int_\Omega \Big(1+\gamma(v)\Big) w+(\bar{u}^2+\alpha \bar{u}\bar{w})\int_\Omega \gamma(v)\\
&\leq \frac{2\alpha+3\alpha\gamma_{2}}{|\Omega|}\|u\|_{L^1}\|w\|_{L^1}
+\frac{\gamma_{2}}{|\Omega|}\|u\|_{L^1}^2,
\end{split}
\end{equation}
in which we have used the fact $\int_{\Omega}(\varphi-\bar{\varphi})^{2}\leq\int_{\Omega}\varphi^{2}$ for all $\varphi\in L^{2}(\Omega)$. We know from Lemma \ref{UW} that $\|u\|_{L^1}\leq\|u_{0}\|_{L^{1}}+\alpha\|w_{0}\|_{L^{1}}\leq|\Omega|(\|u_{0}\|_{L^{\infty}}+\alpha \|w_{0}\|_{L^{\infty}})$ and $\|w\|_{L^1}\leq|\Omega|\|w\|_{L^{\infty}}\leq|\Omega|\|w_{0}\|_{L^{\infty}}$. Therefore, \eqref{0-4} shows
\begin{equation}\label{0-5}
	\frac{d}{dt}\int_{\Omega} |\mathcal{A}^{-\frac{1}{2}}(u+\alpha w-\bar{u}-\alpha\bar{w})|^2+2\gamma_{1}\int_{\Omega} (u-\bar{u})^2+2\alpha^2\int_{\Omega} (w-\bar{w})^2\leq c_{1}(1+\alpha)^{2},
\end{equation}
where $c_{1}=4(1+2\gamma_{2})|\Omega|(\|u_{0}\|_{L^{\infty}}+\|w_{0}\|_{L^{\infty}})^{2}$. Because of $\int_\Omega \A^{-\frac{1}{2}}(u+\alpha w-\bar{u}-\alpha\bar{w})=0$, we can apply the Poincar\'{e} inequality with a positive constant $c_{2}$ and the fact $\|w\|_{L^\infty}\leq \|w_0\|_{L^\infty}$ to obtain
\begin{equation}\label{0-5*}
\begin{split}
	&\int_{\Omega} |\A^{-\frac{1}{2}}(u+\alpha w-\bar{u}-\alpha\bar{w})|^2\\
	&\leq c_2\int_{\Omega}|\nabla \A^{-\frac{1}{2}}(u+\alpha w-\bar{u}-\alpha\bar{w})|^2\\
	&=c_2\int_{\Omega}|u+\alpha w-\bar{u}-\alpha\bar{w}|^2\\
	&\leq 2c_2\int_{\Omega} (u-\bar{u})^2+2c_2\alpha^2\int_{\Omega} (w-\bar{w})^2\\
	&\leq 2c_2\int_{\Omega} (u-\bar{u})^2+2c_{2}\alpha^2|\Omega|\|w_{0}\|_{L^{\infty}}^{2}.
\end{split}
\end{equation}
Substituting \eqref{0-5*} into \eqref{0-5}, and letting $X(t):=\int_\Omega |\A^{-\frac{1}{2}}(u+\alpha w-\bar{u}-\alpha\bar{w})|^2$, one yields
\begin{equation}\label{0-6}
	X'(t)+\frac{\gamma_1}{2c_{2}}X(t)+\gamma_{1}\int_{\Omega} (u-\bar{u})^2\leq c_{3}(1+\alpha)^{2},
\end{equation}
where $c_{3}=c_{1}+\gamma_1|\Omega|\|w_{0}\|_{L^{\infty}}^{2}$. Then applying the Gr\"{o}nwall's inequality to \eqref{0-6}, we first obtain
\begin{equation}\label{0-7}
	X(t)=\int_\Omega |\A^{-\frac{1}{2}}(u+\alpha w-\bar{u}-\alpha\bar{w})|^2\leq c_{4}(1+\alpha)^{2},
\end{equation}
where $c_{4}=\frac{2c_{2}c_{3}}{\gamma_{1}}+2c_{2}|\Omega|(\|u_{0}\|_{L^{\infty}}+\|w_{0}\|_{L^{\infty}})^{2}$. Then integrating \eqref{0-6} over $(t,t+\tau)$ with $\tau:=\min \Big\{ 1,\frac{1}{2} T_{max}\Big\}$ and using \eqref{0-7}, one has
\begin{equation}\label{0-8}
\int_t^{t+\tau}\int_{\Omega} (u-\bar{u})^2\leq \frac{c_{3}\tau+c_{4}}{\gamma_1}(1+\alpha)^{2}\leq \frac{c_{3}+c_{4}}{\gamma_1}(1+\alpha)^{2}.
\end{equation}
By the fact $\int_{\Omega}(u-\bar{u})^{2}=\int_{\Omega}u^{2}-\int_{\Omega}\bar{u}^{2}$, it follows from \eqref{0-8} that
\begin{equation*}
\begin{split}
\int_t^{t+\tau}\int_{\Omega}u^2
= \int_t^{t+\tau}\int_{\Omega}(u-\bar{u})^2+\int_t^{t+\tau}\int_{\Omega}\bar{u}^2\leq \frac{c_{3}+c_{4}}{\gamma_1}(1+\alpha)^{2}+\bar{u}^2|\Omega|\tau,
\end{split}
\end{equation*}
which yields \eqref{2L1-2} by using the fact $\bar{u}\leq\|u_{0}\|_{L^{\infty}}+\alpha\|w_{0}\|_{L^{\infty}}$.

{\bf Case 2: $\theta>0$.} In this case, we let $\mathcal{B}$ denote the self-adjoint realization of $-\Delta+\delta$ under homogeneous Neumann boundary conditions in $L^2(\Omega)$, where $0<\delta<\frac{\theta}{\gamma_{2}}$. Then there exists a constant $c_5>0$ such that
\begin{equation}\label{A1}
		\|\mathcal{B}^{-1}\psi\|_{L^2}\leq c_5\|\psi\|_{L^2} \ \ \ \mathrm{ for \ all}\ \psi\in L^2(\Omega)
		\end{equation}
		and
		\begin{equation}\label{A2}
		\|\mathcal{B}^{-\frac{1}{2}}\psi\|_{L^2}^2=\int_{\Omega} \psi\cdot \mathcal{B}^{-1}\psi \leq c_5\|\psi\|_{L^2}^2 \ \ \ \mathrm{ for \ all}\ \psi\in L^2(\Omega),
		\end{equation}
one can see the details in \cite{LW-CPDE-2015}.
		From the system \eqref{1-1}, we have
		\begin{equation*}
		(u+\alpha w)_t=\Delta(\gamma(v)u+\alpha w)-\theta u,
		\end{equation*}
		which can be rewritten as
		\begin{equation}\label{A3}
		\begin{split}
		(u+\alpha w)_t+\mathcal{B}(\gamma(v)u+\alpha w)
		=\delta {\color{black}(\gamma(v)u+\alpha w)}-\theta u
		=(\delta \gamma(v)-\theta) u+\delta\alpha w.
		\end{split}
		\end{equation}
		With the fact $0<\delta<\frac{\theta}{\gamma_{2}}$ and the boundedness of $w$, we derive
		\begin{equation}\label{A4}
		(\delta \gamma (v)-\theta) u+\delta\alpha w\leq (\delta \gamma_{2}-\theta)u+\delta\alpha\|w_{0}\|_{L^{\infty}}\leq c_{6}\alpha,
		\end{equation}
		where $c_{6}=\frac{\theta\|w_{0}\|_{L^{\infty}}}{\gamma_{2}}$. Hence, multiplying \eqref{A3} by $\mathcal{B}^{-1} (u+\alpha w)\geq 0$, and using \eqref{A4}, one has
		\begin{equation*}
		\frac{1}{2}\frac{d}{dt}\int_{\Omega} |\mathcal{B}^{-\frac{1}{2}} (u+\alpha w)|^2
		+\int_{\Omega} (\gamma(v) u+\alpha w) (u+\alpha w)\leq c_{6}\alpha\int_{\Omega} \mathcal{B}^{-1}(u+\alpha w),
		\end{equation*}
		and hence
		\begin{equation}\label{A5}
		\frac{d}{dt}\int_{\Omega} |\mathcal{B}^{-\frac{1}{2}} (u+\alpha w)|^2
		+2c_7\int_{\Omega} (u+\alpha w)^2\leq 2c_6\alpha\int_{\Omega} \mathcal{B}^{-1}(u+\alpha w),
		\end{equation}
with $c_7:=\min\{\gamma_1,1\}$. Using \eqref{A1} and \eqref{A2}, we can derive that
\begin{equation}\label{A6}
\begin{split}
	&\frac{c_7}{2c_5}\int_{\Omega} |\mathcal{B}^{-\frac{1}{2}} (u+\alpha w)|^2
	+2c_6\alpha\int_{\Omega}\mathcal{B}^{-1}(u+\alpha w)\\
	&\leq \frac{c_7}{2}\int_{\Omega} (u+\alpha w)^2
	+2c_{5}c_{6}\alpha|\Omega |^\frac{1}{2}\|u+\alpha w\|_{L^2}\\
	&\leq c_7\int_{\Omega}(u+\alpha w)^2+\frac{2c_{5}^2c_{6}^{2}|\Omega|}{c_7}\alpha^{2}.
\end{split}
\end{equation}
	Substituting \eqref{A6} into \eqref{A5}, and defining $Y(t):=\int_{\Omega} |\mathcal{B}^{-\frac{1}{2}} (u+\alpha w)|^2$, one has
\begin{equation*}\label{A7}
	Y'(t)+\frac{c_7}{2c_5}Y(t)+c_7\int_{\Omega}(u+\alpha w)^2\leq \frac{2c_5^2c_{6}^{2}|\Omega|}{c_7}\alpha^{2},
\end{equation*}
which combined with the Gr\"{o}nwall's inequality gives
$$Y(t)\leq c_{5}|\Omega|\left((\|u_{0}\|_{L^{\infty}}+\|w_{0}\|_{L^{\infty}})^{2}+\frac{4c_{5}^{2}c_{6}^{2}}{c_{7}^{2}}\right)(1+\alpha)^{2}:=c_{8}(1+\alpha)^{2}$$ and thus
\begin{equation*}
	\int_{t}^{t+\tau}\int_{\Omega} u^2\leq \int_{t}^{t+\tau}\int_{\Omega} (u+\alpha w)^2
	\leq \frac{Y(t)}{c_7}+\frac{2c_{5}^2c_{6}^{2}|\Omega|\tau}{c_7^2}\alpha^{2}\leq c_{9}(1+\alpha)^{2},
\end{equation*}
where $c_{9}=\frac{c_{8}}{c_{7}}+\frac{2c_{5}^{2}c_{6}^{2}|\Omega|\tau}{c_{7}^{2}}$, which gives \eqref{2L1-2}. Then we complete the proof of this lemma.
\end{proof}
	
	\begin{lemma}\label{Gv}
		Let the conditions in Theorem \ref{GB} hold. Then there exist two positive constants $K_2,K_3$ independent of $D$, $\alpha$ and $t$ such that
		\begin{equation}\label{gv}
		\int_{\Omega}|\nabla v|^{2}\leq K_{2}(1+\alpha)^{2}\left(1+\frac{1}{D}\right)\quad\mathrm{for\ all}~~t\in(0,T_{max}),
		\end{equation}
		and
		\begin{equation}\label{gv2-1}
		\int_{t}^{t+\tau}\int_{\Omega}|\Delta v|^{2}\leq K_3(1+\alpha)^{2}\left(1+\frac{1}{D}\right)^2 \quad\mathrm{for\ all}~~t\in(0,\widetilde{T}_{max}).
		\end{equation}
	\end{lemma}
	\begin{proof}
		We multiply the second equation of \eqref{1-1} by $-\Delta v$ and integrate the result with Cauchy-Schwarz inequality to get  for all $t\in(0,T_{max})$
		\begin{equation*}
		\begin{split}
		\frac{1}{2}\frac{d}{dt}\int_{\Omega}|\nabla v|^{2}=&-D\int_{\Omega}|\Delta v|^{2}-\int_{\Omega}u\Delta v+\int_{\Omega} v\Delta v\\
		\leq&-\frac{D}{2}\int_{\Omega}|\Delta v|^{2}+\frac{1}{2D}\int_{\Omega}u^{2}-\int_{\Omega}|\nabla v|^{2},
		\end{split}
		\end{equation*}
		which leads to
		\begin{equation}\label{gvL2}
		\frac{d}{dt}\int_{\Omega}|\nabla v|^{2}+D\int_{\Omega}|\Delta v|^{2}+2\int_{\Omega}|\nabla v|^{2}\leq\frac{1}{D}\int_{\Omega}u^{2}.
		\end{equation}
		Letting $y(t)=\int_{\Omega}|\nabla v|^{2}$ and $h(t)=\frac{1}{D}\int_{\Omega}u^{2}$, we have from \eqref{gvL2} that
		\begin{equation}\label{gv2}
		y'(t)+2y(t)\leq h(t)\quad\mathrm{for\ all}~~t\in(0,T_{max}).
		\end{equation}
Then applying Lemma \ref{G1} with the fact $\int_{t}^{t+\tau}h(s)ds\leq \frac{K_1(1+\alpha)^{2}}{D}$ for $t\in(0,\widetilde{T}_{max})$ to \eqref{gv2} gives
		\begin{equation*}
\begin{split}
		\int_{\Omega}|\nabla v|^{2}
&\leq \max\Big\{\|\nabla v_0\|_{L^2}^2+\frac{K_1(1+\alpha)^{2}}{D},\frac{K_1(1+\alpha)^{2}}{2D\tau}+\frac{2 K_1(1+\alpha)^{2}}{D}\Big\}\\
&\leq \|\nabla v_0\|_{L^2}^2+\frac{2K_1(1+\alpha)^{2}}{D}+\frac{K_1(1+\alpha)^{2}}{2D\tau}\\
&\leq\left(\|\nabla v_0\|_{L^2}^2+2K_1+\frac{K_1}{2\tau}\right)\left(1+\frac{1}{D}\right)(1+\alpha)^{2}
\quad\mathrm{for\ all}~~t\in(0,T_{max}),
\end{split}
		\end{equation*}
which yields \eqref{gv} {\color{black}with $K_2=\|\nabla v_{0}\|_{L^{2}}^{2}+2K_{1}+\frac{K_{1}}{2\tau}$}. On the other hand, integrating \eqref{gvL2} over $(t,~t+\tau)$ for $t\in(0,\widetilde{T}_{max})$ and using \eqref{gv}, we can derive that
\begin{equation*}
\begin{split}
D\int_t^{t+\tau}\int_\Omega |\Delta v |^2
&\leq \frac{1}{D}\int_t^{t+\tau}\int_\Omega u^2+\int_\Omega |\nabla v|^2\\
&\leq \frac{K_1}{D}(1+\alpha)^{2}+K_2\left(1+\frac{1}{D}\right)(1+\alpha)^{2},
\end{split}
\end{equation*}
which implies \eqref{gv2-1} {\color{black}with $K_{3}=K_{1}+K_{2}$}.
	\end{proof}

	\begin{lemma}\label{L2-L}
Let the assumptions in Theorem \ref{GB} hold. Then there exist two positive constants $K_4$ and $K_{5}$, which are independent of $D$ and $\alpha$, such that
		\begin{equation}\label{L2}
		\|u(\cdot,t)\|_{L^2}\leq K_4(1+\alpha)e^{K_5(1+\alpha)^{6}(1+\frac{1}{D})^{4}} \ \ \ \mathrm{for\ all}\ \ t\in(0,T_{max}).
		\end{equation}
	\end{lemma}
\begin{proof}
	Multiplying {\color{black}the first equation of} \eqref{1-1} by $u$ and integrating the result with assumptions (H1) and \eqref{Fw} gives
\begin{equation*}
\begin{split}
	\frac{1}{2}\frac{d}{dt}\int_{\Omega}u^{2}
		=&-\int_{\Omega}\nabla u\cdot\nabla(\gamma(v)u)+\alpha\int_{\Omega}F(w)u^2-\theta\int_{\Omega}u^2\\
		\leq&-\int_{\Omega}\gamma(v)|\nabla u|^{2}-\int_{\Omega}\gamma'(v)u\nabla u\cdot\nabla v+\alpha C_{F}\int_{\Omega}u^2\\
		\leq&-\gamma_{1}\int_{\Omega}|\nabla u|^{2}+\eta\int_{\Omega}u|\nabla u||\nabla v|+\alpha C_{F}\int_{\Omega}u^2\\
		\leq&-\frac{\gamma_{1}}{2}\int_{\Omega}|\nabla u|^{2}+\frac{\eta^{2}}{2\gamma_{1}}\int_{\Omega}u^{2}|\nabla v|^{2}+\alpha C_{F}\int_{\Omega}u^2,
		\end{split}
		\end{equation*}
		which yields
		\begin{equation}\label{L2-1}
		\frac{d}{dt}\int_{\Omega}u^2+\gamma_{1}\int_{\Omega}|\nabla u|^{2}\leq\frac{\eta^{2}}{\gamma_{1}}\int_{\Omega}u^2|\nabla v|^{2}+2\alpha C_{F}\int_{\Omega}u^2.
		\end{equation}
Moreover, applying Gagliardo-Nirenberg inequality and Young inequality to the first term on the right hand side of \eqref{L2-1}, we obtain a constant $c_{1}>0$ such that
\begin{equation}\label{L2-2}
\begin{split}
	\frac{\eta^{2}}{\gamma_{1}}\int_{\Omega}u^2|\nabla v|^{2}
	\leq &\frac{\eta^2}{\gamma_1}\|u\|_{L^4}^2\|\nabla v\|_{L^4}^2\\
	\leq &\frac{c_{1}\eta^{2}}{\gamma_{1}}\left(\|\nabla u\|_{L^2}\|u\|_{L^2} +\|u\|_{L^2}^2\right)\left(\|\Delta v\|_{L^2}\|\nabla v\|_{L^2}+\|\nabla v\|_{L^2}^2\right)\\
	\leq &\frac{c_{1}\eta^{2}}{\gamma_{1}}\|\nabla u\|_{L^2}\|u\|_{L^2}\|\Delta v\|_{L^2}\|\nabla v\|_{L^2} +\frac{c_{1}\eta^{2}}{\gamma_{1}}\|\nabla u\|_{L^2}\|u\|_{L^2}\|\nabla v\|_{L^2}^2\\
	&+\frac{c_{1}\eta^{2}}{\gamma_{1}}\|u\|_{L^2}^2\|\Delta v\|_{L^2}\|\nabla v\|_{L^2} +\frac{c_{1}\eta^{2}}{\gamma_{1}}\|u\|_{L^2}^2\|\nabla v\|_{L^2}^2\\
	\leq & \gamma_1 \|\nabla u\|_{L^2}^2+\frac{c_{1}\eta^{2}}{\gamma_{1}}\left(2+\frac{c_{1}\eta^{2}}{2\gamma_{1}^{2}}\|\nabla v\|_{L^2}^2\right)\|u\|_{L^2}^2\|\nabla v\|_{L^2}^2\\
	&+\frac{c_{1}\eta^{2}}{\gamma_{1}}\left(\frac{1}{4}+\frac{c_{1}\eta^{2}}{2\gamma_{1}^{2}}\|\nabla v\|_{L^2}^2\right)\|u\|_{L^2}^2\|\Delta v\|_{L^2}^2.
\end{split}
\end{equation}
Substituting \eqref{L2-2} into \eqref{L2-1}, and using \eqref{gv}, we conclude
		\begin{equation}\label{L2-3}
		\frac{d}{dt}\|u\|_{L^2}^2\leq c_{2}(1+\alpha)^{4}\left(1+\frac{1}{D}\right)^2(1+\|\Delta v\|_{L^{2}}^{2})\|u\|_{L^{2}}^{2},
		\end{equation}
where $c_{2}=\frac{c_{1}\eta^{2}}{\gamma_{1}}\left(2K_{2}+\frac{1}{4}+\frac{c_{1}\eta^{2}K_{2}(1+K_{2})}{2\gamma_{1}^{2}}\right)+2C_{F}$.
On the other hand, using the facts \eqref{2L1-2} and \eqref{gv2-1}, then for any $t\in(0,T_{max})$, we can find a $t_0\geq 0$ satisfying $t_0\in(0,\tilde{T}_{max})$ and $t_0\in((t-\tau)^+,t)$ such that
\begin{equation}\label{L2-4*}
\|u(\cdot,t_0)\|_{L^2}^2\leq c_{3}(1+\alpha)^{2},
\end{equation}
and
\begin{equation}\label{L2-4}
\int_{t_0}^{t_0+\tau}\int_\Omega|\Delta v|^2\leq K_3(1+\alpha)^{2}\left(1+\frac{1}{D}\right)^2,
\end{equation}
with $c_{3}=\|u_{0}\|_{L^{2}}^{2}+\frac{K_{1}}{\tau}$. Then we integrate \eqref{L2-3} over $(t_0,t)$, and use the facts  \eqref{L2-4*}, \eqref{L2-4} and $t\leq t_0+\tau\leq t_0+1$ to obtain
		\begin{equation*}
		\begin{split}
		\|u(\cdot,t)\|_{L^2}^2
		&\leq \|u(\cdot,t_0)\|_{L^2}^2 e^{c_2(1+\alpha)^{4}\left(1+\frac{1}{D}\right)^2\int_{t_0}^t(1+\|\Delta v(\cdot,s)\|_{L^2}^2)ds}\\
		&\leq \|u(\cdot,t_0)\|_{L^2}^2 e^{c_2(1+\alpha)^{4}\left(1+\frac{1}{D}\right)^2\int_{t_0}^t ds+c_2(1+\alpha)^{4}\left(1+\frac{1}{D}\right)^2\int_{t_0}^t\|\Delta v(\cdot,s)\|_{L^2}^2 ds}\\
		&\leq c_{3}(1+\alpha)^{2}e^{c_2(1+\alpha)^{4}\left(1+\frac{1}{D}\right)^2+c_2K_3(1+\alpha)^{6}\left(1+\frac{1}{D}\right)^4},
		\end{split}
		\end{equation*}
		which yields \eqref{L2} {\color{black}with $K_{4}=c_{3}$ and $K_{5}=c_{2}(1+K_{3})$}. Then we finish the proof of this lemma.
	\end{proof}
%
\begin{lemma}\label{L4-u}
Suppose the conditions in {\color{black}Theorem \ref{GB}} hold. Let $(u,v,w)$ be the solution of the system \eqref{1-1}. Then it holds that
\begin{equation}\label{L4-u1}
\|u(\cdot,t)\|_{L^4}\leq K_{6}(1+\alpha)^{3}\left(1+\frac{1}{D}\right)^{2}e^{3K_{5}(1+\alpha)^{6}(1+\frac{1}{D})^{4}},\quad\mathrm{for\ all}~~t\in(0,T_{max}),
\end{equation}
where $K_6>0$ is a constant independent of $\alpha$, $D$ and $t$.
\end{lemma}
\begin{proof}
	With the fact that $0\leq F(w)\leq C_{F}$ from \eqref{Fw} and the assumptions (H1), we multiply the first equation of \eqref{1-1} with $u^{3}$ and integrate the result to have
	\begin{equation*}\label{uLk1}
	\begin{split}
		\frac{1}{4}\frac{d}{dt}\int_{\Omega}u^{4}
		=&-3\int_{\Omega}u^{2}\nabla u\cdot\nabla(\gamma(v)u)+\alpha\int_{\Omega}F(w)u^{4}-\theta\int_{\Omega}u^{4}\\
		\leq&-3\int_{\Omega}\gamma(v)u^2|\nabla u|^{2}-3\int_{\Omega}\gamma'(v)u^{3}\nabla u\cdot\nabla v+\alpha C_{F}\int_{\Omega}u^{4}\\
		\leq&-3\gamma_{1}\int_{\Omega}u^{2}|\nabla u|^{2}+3\eta\int_{\Omega}u^{3}|\nabla u||\nabla v|+\alpha C_{F}\int_{\Omega}u^{4}\\
		\leq&-\frac{3\gamma_{1}}{2}\int_{\Omega}u^{2}|\nabla u|^{2} +\frac{3\eta^{2}}{2\gamma_{1}}\int_{\Omega}u^{4}|\nabla v|^{2} +\alpha C_{F}\int_{\Omega}u^{4}{\color{black},}
	\end{split}
	\end{equation*}
which yields that
\begin{equation}\label{Lu4-1}
\frac{d}{dt}\int_\Omega u^4+\frac{3\gamma_1}{2}\int_\Omega |\nabla u^2|^2\leq \frac{6\eta^{2}}{\gamma_{1}}\int_{\Omega}u^4|\nabla v|^{2}+4\alpha C_{F}\int_{\Omega}u^{4}.
\end{equation}
Using Gagliardo-Nirenberg inequality and Young's inequality, along with the facts \eqref{L4v} and $\|u^2\|_{L^1}=\|u\|_{L^2}^2$, we can find a constant $c_{1}>0$ independent of $\alpha$ and $D$, such that
\begin{equation}\label{Lu4-2}
\begin{split}
\frac{6\eta^{2}}{\gamma_{1}}\int_{\Omega}u^{4}|\nabla v|^{2}
&\leq \frac{6\eta^{2}}{\gamma_{1}}\left(\int_{\Omega}u^{8}\right)^\frac{1}{2}\left(\int_\Omega |\nabla v|^{4}\right)^\frac{1}{2}\\
&=\frac{6\eta^{2}}{\gamma_{1}}\|u^2\|_{L^4}^{2}\|\nabla v\|_{L^4}^{2}\\
&\leq \frac{6\eta^{2}c_{1}}{\gamma_{1}}\left(\|\nabla u^2\|_{L^2}^{\frac{3}{2}}\|u^2\|_{L^1}^{\frac{1}{2}}+\|u^2\|_{L^1}^2\right)\|\nabla v\|_{L^4}^{2}\\
&\leq \frac{6\eta^{2}c_{1}}{\gamma_{1}}\|\nabla u^2\|_{L^2}^{\frac{3}{2}}\|u\|_{L^2}\|\nabla v\|_{L^4}^{2}+\frac{6\eta^{2}c_{1}}{\gamma_{1}}\|u\|_{L^2}^{4}\|\nabla v\|_{L^4}^{2}\\
&\leq \gamma_1\|\nabla u^2\|_{L^2}^2+c_{2}\|\nabla v\|_{L^4}^2\|u\|_{L^2}^4(\|\nabla v\|_{L^4}^6+1),
\end{split}
\end{equation}
where
\begin{equation*}
c_2:=\frac{1}{4}\left(\frac{3}{4\gamma_1}\right)^3
\left(\frac{6\eta^2c_1}{\gamma_1}\right)^4+\frac{6\eta^{2}c_1}{\gamma_{1}}.
\end{equation*}
Furthermore, using the Gagliardo-Nirenberg inequality and Young's inequality again, we can find a constant $c_3>0$ independent of $D$ and $\alpha$, such that
\begin{equation}\label{Lu4-3}
\begin{split}
(1+4\alpha C_{F})\int_\Omega u^4&=(1+4\alpha C_{F})\|u^2\|_{L^2}^2\\
&\leq c_3(1+4\alpha C_{F})\left(\|\nabla u^2\|_{L^2}^\frac{3}{2}\|u^2\|_{L^\frac{1}{2}}^{\frac{1}{2}}+\|u^2\|_{L^\frac{1}{2}}^2\right)\\
&\leq c_3(1+4\alpha C_{F})\left(\|\nabla u^2\|_{L^2}^\frac{3}{2}\|u\|_{L^1}+\|u\|_{L^1}^{4}\right)\\
&\leq \frac{\gamma_1}{2}\|\nabla u^2\|_{L^2}^2+c_4(1+\alpha)^{8},
\end{split}
\end{equation}
where $c_{4}=\left(\frac{c_{3}^{4}}{4}(4C_{F}+1)^{4}\left(\frac{3}{2\gamma_{1}}\right)^{3}+c_{3}(4C_{F}+1)\right)(\|u_{0}\|_{L^{1}}+\|w_{0}\|_{L^{1}})^{4}$. Substituting \eqref{Lu4-2} and \eqref{Lu4-3} into \eqref{Lu4-1},  one has
\begin{equation}\label{uL4}
\begin{split}
\frac{d}{dt}\int_\Omega u^4+\int_\Omega u^4
&\leq c_2\|\nabla v\|_{L^4}^2\|u\|_{L^2}^4(\|\nabla v\|_{L^4}^6+1)+c_4(1+\alpha)^{8}\\
\end{split}
\end{equation}
By the scaling $\tilde{t}=D t$, and applying the variation-of-constants formula to the second equation of \eqref{1-1}, one has
\begin{equation}\label{vq-1}
v(\cdot,\tilde{t})=e^{(\Delta -\frac{1}{D})\tilde{t}}v_0+\frac{1}{D}\int_0^{\tilde{t}}
e^{(\Delta-\frac{1}{D})(\tilde{t}-s)}u(\cdot,s)ds.
\end{equation}
Then using the semigroup estimates \eqref{Lp-2} and \eqref{Lp-3}, we derive from \eqref{vq-1}
\begin{equation*}
\begin{split}
{\color{black}\|\nabla v(\cdot,\tilde{t})\|_{L^4}}
&\leq \|\nabla e^{(\Delta -\frac{1}{D})\tilde{t}}v_0 \|_{L^4}+\frac{1}{D}\int_0^{\tilde{t}}\|\nabla e^{(\Delta-\frac{1}{D})(\tilde{t}-s)}u(\cdot,s)\|_{L^4}ds\\
&\leq k_1e^{-\lambda_1\tilde{t}}\|\nabla v_0\|_{L^4}
+\frac{k_2}{D}\int_0^{\tilde{t}}
\left(1+(\tilde{t}-s)^{-\frac{3}{4}}\right)
e^{-\lambda_1(\tilde{t}-s)}\|u(\cdot,s)\|_{L^2}ds\\
&\leq k_1\|\nabla v_0\|_{L^4}+\frac{k_{2}K_{4}}{D\lambda_1}\left(1+\Gamma(1/4)\lambda_{1}^{\frac{3}{4}}\right)(1+\alpha)e^{K_{5}(1+\alpha)^{6}(1+\frac{1}{D})^{4}},
\end{split}
\end{equation*}
  which gives
  \begin{equation}\label{L4v}
\|\nabla v(\cdot,t)\|_{L^4}\leq c_{5}(1+\alpha)\left(1+\frac{1}{D}\right)e^{K_{5}(1+\alpha)^{6}(1+\frac{1}{D})^{4}},
\end{equation}
with $c_5=k_{1}\|\nabla v_{0}\|_{L^{4}}+\frac{k_{2}K_{4}}{\lambda_{1}}	\left(1+\Gamma(1/4)\lambda_{1}^{\frac{3}{4}}\right)$. Then substituting \eqref{L4v} into \eqref{uL4}, one can find a constant $c_6:=c_{2}K_{4}^{4}c_{5}^{2}(c_{5}^{6}+1)+c_{4}$ to obtain
\begin{equation*}\label{uL4}
\begin{split}
\frac{d}{dt}\int_\Omega u^4+\int_\Omega u^4
&\leq c_{6}(1+\alpha)^{12}\left(1+\frac{1}{D}\right)^{8}e^{12K_{5}(1+\alpha)^{6}(1+\frac{1}{D})^{4}}.
\end{split}
\end{equation*}
 This along with the Gr\"{o}nwall's inequality yields a constant $c_7=c_{6}+\|u_{0}\|_{L^{4}}^{4}$ independent of $D$ and $\alpha$ so that
\begin{equation*}
\begin{split}
\|u(\cdot,t)\|_{L^4}^4
&\leq\|u_0\|_{L^4}^4+ c_{6}(1+\alpha)^{12}\left(1+\frac{1}{D}\right)^{8}e^{12K_{5}(1+\alpha)^{6}(1+\frac{1}{D})^{4}}\\
&\leq c_{7}(1+\alpha)^{12}\left(1+\frac{1}{D}\right)^{8}e^{12K_{5}(1+\alpha)^{6}(1+\frac{1}{D})^{4}},
\end{split}
\end{equation*}
which yields \eqref{L4-u1}.
\end{proof}

\begin{lemma}\label{LuI}
Let the conditions in Lemma \ref{L4-u} hold. Suppose $(u,v,w)$ is a solution of \eqref{1-1}. Then it follows that
\begin{equation}\label{LuI-1}
\|u(\cdot,t)\|_{L^\infty}\leq K_7(1+\alpha)^{13}\left(1+\frac{1}{D}\right)^{12}e^{12K_{5}(1+\alpha)^{6}(1+\frac{1}{D})^{4}}\quad\mathrm{for\ all}~~t\in(0,T_{max}),
\end{equation}
where the constant $K_7>0$ is independent of $D$ and $\alpha$.
\end{lemma}
\begin{proof}
Using \eqref{Lp-2}, \eqref{L4-u1} and the estimate
$
\|\nabla e^{\tilde{t}\Delta}v_0\|_{L^\infty}\leq c_{1}\|v_0\|_{W^{1,\infty}} \ \mathrm{for\ all}\ \tilde{t}>0$(see \cite{Winkler-DCDSA-2016}), from \eqref{vq-1} we have
\begin{equation*}\label{LvI}
\begin{split}
\|\nabla v(\cdot,{\color{black}\tilde{t}})\|_{L^\infty}
&\leq \|\nabla e^{(\Delta -\frac{1}{D})\tilde{t}}v_0 \|_{L^\infty}+\frac{1}{D}\int_0^{\tilde{t}}\|\nabla e^{(\Delta-\frac{1}{D})(\tilde{t}-s)}u(\cdot,s)\|_{L^\infty}ds\\
& \leq c_1\|v_0\|_{W^{1,\infty}}+\frac{k_2}{D}\int_0^{\tilde{t}}
{\color{black}\left(1+(\tilde{t}-s)^{-\frac{3}{4}}\right)}e^{-\lambda_1(\tilde{t}-s)}
\|u(\cdot,s)\|_{L^4}ds\\
&\leq c_1\|v_0\|_{W^{1,\infty}}+\frac{k_2K_{6}}{D\lambda_1}\left(1+\Gamma(1/4)\lambda_{1}^{\frac{3}{4}}\right)(1+\alpha)^{3}\left(1+\frac{1}{D}\right)^{2}
e^{3K_{5}(1+\alpha)^{6}(1+\frac{1}{D})^{4}}
\end{split}
\end{equation*}
which implies
\begin{equation}\label{LuI-2}
\|\nabla v(\cdot,t)\|_{L^\infty}\leq c_2(1+\alpha)^{3}\left(1+\frac{1}{D}\right)^{3}e^{3K_{5}(1+\alpha)^{6}(1+\frac{1}{D})^{4}},
\end{equation}
where $c_2:=c_1\|v_0\|_{W^{1,\infty}}+\frac{k_2K_{6}}{\lambda_1}\left(1+\Gamma(1/4)\lambda_{1}^{\frac{3}{4}}\right)${\color{black}. With \eqref{Fw}  and \eqref{LuI-2}}, we multiply the first equation of \eqref{1-1} {\color{black}by} $u^{p-1}(p\geq 2)$ and integrate the result to obtain
\begin{equation}\label{LuI-3}
\begin{split}
	\frac{1}{p}\frac{d}{dt}\int_{\Omega}u^{p}
	=&-(p-1)\int_{\Omega}u^{p-2}\nabla u\cdot\nabla(\gamma(v)u) +\alpha\int_{\Omega}F(w)u^{p}-\theta\int_{\Omega}u^{p}\\
	\leq&-(p-1)\int_{\Omega}\gamma(v)u^{p-2}|\nabla u|^{2}-(p-1)\int_{\Omega}\gamma'(v)u^{p-1}\nabla u\cdot\nabla v+\alpha C_{F}\int_{\Omega}u^{p}\\
	\leq&-\gamma_{1}(p-1)\int_{\Omega}u^{p-2}|\nabla u|^{2} +\eta(p-1)\int_{\Omega}u^{p-1}|\nabla u||\nabla v|+\alpha C_{F}\int_{\Omega}u^{p}\\
	\leq&-\frac{\gamma_{1}(p-1)}{2}\int_{\Omega}u^{p-2}|\nabla u|^{2} +\frac{\eta^{2}}{2\gamma_{1}}(p-1)\int_{\Omega}u^{p}|\nabla v|^{2} +\alpha C_{F}(p-1)\int_{\Omega}u^{p}\\
	\leq&-\frac{\gamma_{1}(p-1)}{2}\int_{\Omega}u^{p-2}|\nabla u|^{2} +\mathcal{K}_{D}(p-1)\int_\Omega u^{p},
\end{split}
\end{equation}
where $\mathcal{K}_{D}$ is independent of $p$ and defined by
\begin{equation*}\label{KD}
\mathcal{K}_{D}:=\left(\frac{\eta^2c_2^2}{2\gamma_1}+C_{F}\right)(1+\alpha)^{6}\left(1+\frac{1}{D}\right)^{6}e^{6K_{5}(1+\alpha)^{6}(1+\frac{1}{D})^{4}}.
\end{equation*}
Then using the identity $\int_\Omega u^{p-2}|\nabla u|=\frac{4}{p^2}\int_\Omega|\nabla u^\frac{p}{2}|$, from \eqref{LuI-3} one has
\begin{equation}\label{LuI-4}
\begin{split}
	\frac{d}{dt}\int_\Omega u^p+p(p-1)\int_\Omega u^p\leq&-\frac{2\gamma_1(p-1)}{p}\int_\Omega |\nabla u^\frac{p}{2}|^2+(\mathcal{K}_{D}+1)p(p-1)\int_\Omega u^p.
\end{split}
\end{equation}
 Using the interpolation inequality and Young's inequality with $\varepsilon$, then for all $f\in W^{1,2}(\Omega)$, one has
 \begin{equation}\label{LuI-5}
 \|f\|_{L^2}^2\leq \varepsilon \|\nabla f\|_{L^2}^2+c_{3}(1+\varepsilon^{-1})\|f\|_{L^1}^2
 \end{equation}
 for any $\varepsilon>0$, where $c_{3}>0$ only depends on $\Omega$. Then letting $f=u^\frac{p}{2}$ and $\varepsilon=\frac{2\gamma_1}{p^2(\mathcal{K}_D+1)}$ in \eqref{LuI-5}, we can derive that
\begin{equation}\label{LuI-6}
\begin{split}
	(\mathcal{K}_{D}+1)p(p-1)\int_\Omega u^p\leq \frac{2\gamma_1(p-1)}{p}\int_\Omega |\nabla u^\frac{p}{2}|^2+ \widetilde{\mathcal{K}}_Dp(p-1)(1+p^2)\left(\int_\Omega u^\frac{p}{2}\right)^2{\color{black},}
\end{split}
\end{equation}
 where
\begin{equation*}
\begin{split}
	\widetilde{\mathcal{K}}_D=&\frac{c_{3}(1+2\gamma_{1})}{2\gamma_{1}}(\mathcal{K}_{D}+1)^{2}\\
	=&\frac{c_{3}(1+2\gamma_{1})}{2\gamma_{1}}\left(\frac{\eta^2c_2^2}{2\gamma_1}+C_{F}+1\right)^{2}(1+\alpha)^{12}\left(1+\frac{1}{D}\right)^{12}e^{12K_{5}(1+\alpha)^{6}(1+\frac{1}{D})^{4}}{\color{black}.}
\end{split}
\end{equation*}
Substituting \eqref{LuI-6} into \eqref{LuI-4} and using the fact $1+p^2\leq(1+p)^2$, one has
 \begin{equation*}
 \frac{d}{dt}\int_\Omega u^p+p(p-1)\int_\Omega u^p\leq \widetilde{\mathcal{K}}_Dp(p-1)(1+p)^2\left(\int_\Omega u^{\frac{p}{2}}\right)^{2},
 \end{equation*}
 which gives
\begin{equation}\label{LuI-7}
\int_\Omega u^p(x,t)\leq \int_\Omega u_0^p(x)+ \widetilde{\mathcal{K}}_D(1+p)^2\sup\limits_{0\leq t\leq T_{max}}\left(\int_\Omega u^\frac{p}{2}(x,t)\right)^2.
\end{equation}
Then using the Moser iteration \cite{Alikakos-1979}( see also the similar argument as in \cite{T-JMAA2-2011,TW-M3AS-2013}), from \eqref{LuI-7} one has
\begin{equation*}
 \|u(\cdot,t)\|_{L^\infty}\leq 2^6\widetilde{\mathcal{K}}_D(1+|\Omega|)(1+\alpha)(\|u_0\|_{L^\infty}+\|w_0\|_{L^{\infty}}),
 \end{equation*}
which gives \eqref{LuI-1}.

\end{proof}

	\begin{proof}[Proof of Theorem $\ref{GB}$] For any fixed $D>0$ and $\alpha\geq0$, from Lemma \ref{LuI}, we can find a constant $C>0$ independent of $t$ such that
		\begin{equation*}
		\|u(\cdot,t)\|_{L^\infty}\leq C(1+\alpha)^{13}\left(1+\frac{1}{D}\right)^{12}e^{12K_{5}(1+\alpha)^{6}(1+\frac{1}{D})^{4}},
		\end{equation*}
	which combined with the local existence results in Lemma \ref{LS} {\color{black}proves} Theorem \ref{GB}.
	\end{proof}
	
	\section{Asymptotic behavior (Proof of Theorem $\ref{LTB}$)}
	In this section, we will derive the asymptotic behavior of solutions as shown in Theorem $\ref{LTB}$. 
	Before embarking on these details, we first use the standard parabolic property to improve the regularity of $u$, $v$ and $w$ as follows.
	\begin{lemma}\label{A10*} Let $(u,v,w)$ be the nonnegative global classical solution of $\eqref{1-1}$ obtained in Theorem \ref{GB}. Then there exist $\sigma\in(0,1)$ and $C>0$ such that
		\begin{equation}\label{A10*-1}
		\|u(\cdot,t)\|_{C^{\sigma,\frac{\sigma}{2}}(\bar{\Omega}\times[t,t+1])}\leq C\quad\mathrm{for\ all}\quad t>1
		\end{equation}
		and
		\begin{equation}\label{A10*-2}
		\|v(\cdot,t)\|_{C^{2+\sigma,1+\frac{\sigma}{2}}(\bar{\Omega}\times[t,t+1])}+\|w\|_{C^{2+\sigma,1+\frac{\sigma}{2}}(\bar{\Omega}\times[t,t+1])}\leq C\quad\mathrm{for\ all}\quad t>1.
		\end{equation}
	\end{lemma}
{\color{black}
\begin{proof}
Let $A(x,t,u,\nabla u)=\gamma(v)\nabla u+\gamma'(v)u\nabla v$ and $B(x,t,u)=\alpha F(w)u-\theta u$. Then we can rewrite the first equation of \eqref{1-1}  as follows
\begin{equation*}
	u_{t}=\nabla\cdot A(x,t,u,\nabla u)+B(x,t,u).
\end{equation*}
Noting that Theorem \ref{GB} gives two positive constants $c_{1}$ and $c_{2}$ satisfying $\|u\|_{L^{\infty}}\leq c_{1}$ and $\|v\|_{W^{1,\infty}}+\|w\|_{W^{1,\infty}}\leq c_{2}$, we end up with
\begin{equation}\label{HA1}
\begin{split}
	A(x,t,u,\nabla u)\cdot\nabla u&=\left(\gamma(v)\nabla u+u\gamma'(v)\nabla v\right)\cdot\nabla u\\
	&\leq\gamma(v)|\nabla u|^{2}+\gamma'(v)u\nabla u\cdot\nabla v\\
	&\leq\frac{\gamma(v)}{2}|\nabla u|^{2}+\frac{(\gamma'(v))^{2}}{2\gamma(v)}u^{2}|\nabla v|^{2}\\
	&\leq\frac{\gamma_{2}}{2}|\nabla u|^{2}+\frac{c_{1}^{2}c_{2}^{2}\eta^{2}}{2\gamma_{1}}
\end{split}
\end{equation}
and
\begin{equation}\label{HA2}
\begin{split}
	|A(x,t,u,\nabla u)|&=|\gamma(v)\nabla u+\gamma'(v)u\nabla v|\\
	&\leq|\gamma(v)||\nabla u|+|\gamma'(v)|\|u\|_{L^{\infty}}\|\nabla v\|_{L^{\infty}}\\
	&\leq\gamma_{2}|\nabla u|+c_{1}c_{2}\eta.
\end{split}
\end{equation}
Moreover, since \eqref{Fw} guarantees $F(w)\leq C_{F}$ and hence
\begin{equation}\label{HA3}
\begin{split}
	|B(x,t,u)|&=|\alpha F(w)u-\theta u|\\
	&\leq\alpha|F(w)|\|u\|_{L^{\infty}}+\theta\|u\|_{L^{\infty}}\\
	&\leq c_{1}(\alpha C_{F}+\theta).
\end{split}
\end{equation}
With \eqref{HA1}--\eqref{HA3} in hand, we obtain \eqref{A10*-1} by applying \cite[Theorem 1.3]{RT}. Furthermore, the standard parabolic regularity combined with \eqref{A10*-1} infers \eqref{A10*-2} directly.
\end{proof}
}
	\subsection{Case of $\theta>0$}
	In this subsection, we are devoted to studying the large time behavior of solutions for the case $\theta>0$. Notice that $\int_0^\infty\int_{\Omega} u<\infty$ and the relative compactness of $(u(\cdot,t))_{t>1}$ in $C(\Omega)$ (see Lemma $\ref{A10*}$) indicate some decay information for $u$ and hence the decay properties of $v$ from the second equation of $\eqref{1-1}$. Precisely, we have the following results.
	\begin{lemma}\label{cu} Let the conditions in {\color{black}Theorem \ref{LTB}} hold, and suppose $\theta>0$ and $(u,v,w)$ is the solution of the system \eqref{1-1}. Then it follows that
		\begin{equation}\label{cu-1}
		\|u(\cdot,t)\|_{L^\infty}\to 0 \ \ \ as \ \ t\to\infty,
		\end{equation}
		and
		\begin{equation}\label{cv-1}
		\|v(\cdot,t)\|_{L^\infty}\to 0 \ \ \ as \ \ t\to\infty.
		\end{equation}
	\end{lemma}
	\begin{proof}
First, we claim that
		\begin{equation}\label{cu-4}
		u(\cdot,t)\to 0 \ \ \mathrm{in} \ \ L^1(\Omega) \ \ \mathrm{as} \ \ t\to\infty.
		\end{equation}
	Indeed, defining $A(t):=\int_{\Omega} u>0$, we have $\int_0^\infty |A(t)|=\int_0^\infty\int_{\Omega} u<\infty$ from $\eqref{L1u}$. Furthermore, from the first equation of \eqref{1-1} and the fact $ w\in L^\infty(\Omega)$ (see Lemma \ref{UW}), we can  derive that
\begin{equation*}
	\int_0^\infty|A'(t)|=\int_0^\infty\Big|\int_{\Omega} (\alpha F(w)-\theta) u\Big|\leq c_1\int_0^\infty\int_{\Omega} u<\infty,
\end{equation*}
	 which together with the fact $\int_0^\infty |A(t)|=\int_0^\infty\int_{\Omega} u<\infty$ gives $A(t)\to 0$ as $t\to\infty$. This verifies the claim \eqref{cu-4}.

 With \eqref{cu-4} in hand, we shall show \eqref{cu-1} holds. In fact, if $\eqref{cu-1}$ is false, we can find a constant {\color{black}$c_{2}>0$} and a time sequence $(t_k)_{k\in\mathbb{N}}\subset(1,\infty)$ satisfying $t_k\to\infty$ as $k\to \infty$ such that
		\begin{equation}\label{cu-2}
		\|u(\cdot,t_k)\|_{L^\infty}\geq {\color{black}c_{2}} ~~\mathrm{for~~ all} ~k\in\mathbb{N}.
		\end{equation}
		On the other hand, using \eqref{A10*-1} in Lemma $\ref{A10*}$ and the Arzel\`{a}-Ascoli theorem, we know that $(u(\cdot,t))_{t>1}$ is relatively compact in $C(\Omega)$. Hence, we can extract a subsequence, still denoted by $(t_k)_{k\in\mathbb{N}}\subset(1,\infty)$, such that
		\begin{equation*}\label{cu-3}
		u(\cdot,t_k)\to u_{\infty} \ \ \mathrm{in} \ \ L^\infty(\Omega) \ \ \mathrm{as} \ \ k\to\infty,
		\end{equation*}
	which combined with \eqref{cu-4} implies $u_\infty\equiv0$. This however contradicts $\eqref{cu-2}$ and hence $\eqref{cu-1}$ is proved.

	Next, we show \eqref{cv-1} holds. To this end, we consider the following system
		\begin{equation}\label{cv-2}
		\begin{cases}
		v_t+v=D\Delta v+u,\ \ &x\in\Omega,t>0,\\
		\frac{\partial v}{\partial \nu}=0,\ \ &x\in\partial\Omega,t>0,\\
		v(x,0)=v_0(x),&x\in\Omega.
		\end{cases}
		\end{equation}
		Let $v^*(t)$ be solutions of the ODE problem
		\begin{equation}\label{k2-7}
		\begin{cases}
		v_t^*(t)+v^*(t)=\|u(\cdot,t)\|_{L^\infty},\quad t>0,\\
		v^*(0)=\|v_0\|_{L^\infty}.
		\end{cases}
		\end{equation}
		{\color{black}By the comparison principle, we know that $v^*(t)$ is a super-solution of \eqref{cv-2} satisfying $v(x,t)\leq v^*(t)$ for all $x\in\Omega$, $t>0$.}
		Similarly, we can prove that $v(x,t)\geq -v^*(t)$ for all $x\in\Omega,\ t>0$. Hence, one has
		\begin{equation}\label{k2-9}
		|v(x,t)|\leq v^*(t) \ \ \mathrm{for\ all} \ \ x\in\Omega,\ t>0.
		\end{equation}
		On the other hand, from \eqref{k2-7} and using the fact $\|u(\cdot,t)\|_{L^\infty}\to 0$ as $t\to\infty$ we have
		\begin{equation*}\label{k2-10}
		v^*(t)\to 0 \ \ \mathrm{as}\ \ t\to\infty,
		\end{equation*}
		which combined with \eqref{k2-9} gives
		\begin{equation*}\label{k2-11}
		\begin{split}
		\|v(\cdot,t)\|_{L^\infty}\leq v^*(t) \to 0 \ \ \mathrm{as}\ \ t\to\infty.
		\end{split}
		\end{equation*}
		This yields \eqref{cv-1} and completes the proof of Lemma \ref{cu}.
	\end{proof}
	
	\begin{lemma}\label{Lw*} Suppose the conditions in {\color{black}Lemma \ref{cu}} hold.
Let $(u,v,w)$ be the solution of the system \eqref{1-1}. Then we have the following {\color{black}result}
		\begin{equation}\label{cw-1}
		\|w(\cdot,t)-w_*\|_{L^\infty} \to 0 \ \ \ as \ \ t\to\infty,
		\end{equation}
		where $w_{*}> 0$ is a constant determined by {\color{black} $w_*=\frac{1}{|\Omega|}\|w_0\|_{L^1}-\frac{1}{|\Omega|}\int_0^\infty\int_\Omega uF(w)$.}
	\end{lemma}
	\begin{proof} Let $\bar{w}(t)=\frac{1}{|\Omega|}\int_{\Omega} w=\frac{1}{|\Omega|} \|w\|_{L^1}$, then the third equation of \eqref{1-1} can be rewritten as
		\begin{equation}\label{cw-2}
		\left(w-\bar{w}\right)_t=\Delta(w-\bar{w})-uF(w)+\overline{uF(w)}.
		\end{equation}
Then applying the variation-of-constants formula to $\eqref{cw-2}$, we get
		\begin{equation*}	w(\cdot,t)-\bar{w}(t)=e^{\frac{t}{2}\Delta}\big(w(\cdot,t/2)-\bar{w}(t/2)\big)-\int_{\frac{t}{2}}^te^{(t-s)\Delta}\left(uF(w)-\overline{uF(w)}\right)ds,
		\end{equation*}
		which, together with the fact $\|w(\cdot,t)\|_{L^\infty}\leq c_1$ and \eqref{Lp-1}, gives
		\begin{equation}\label{cw-3}
		\begin{split}
		&\|w(\cdot,t)-\bar{w}(t)\|_{L^\infty}\\
		&\leq\|e^{\frac{t}{2}\Delta}\big(w(\cdot,t/2))-\bar{w}(t/2)\big)\|_{L^\infty}+\int_{\frac{t}{2}}^t\|e^{(t-s)\Delta}(uF(w)-\overline{uF(w)})\|_{L^\infty}ds\\
		&\leq k_{1}e^{-\frac{\lambda_{1}t}{2}}\|w(\cdot,t/2)-\bar{w}(t/2)\|_{L^\infty}+{\color{black}k_{1}C_{F}}\int_{\frac{t}{2}}^t e^{-(t-s)\lambda_1}\|u(\cdot,s)\|_{L^\infty}ds\\
		&\leq 2k_{1}c_{1}e^{-\frac{\lambda_{1}t}{2}}+\frac{{\color{black}k_{1}C_{F}}}{\lambda_{1}}\sup_{\frac{t}{2}\leq s\leq t}\|u(\cdot,s)\|_{L^\infty}.
		\end{split}
		\end{equation}
{\color{black}Then using the decay property of $u$ in \eqref{cu-1}, from \eqref{cw-3} one has
		\begin{equation}\label{cw-4}
		\lim\limits_{t\to\infty}\|w(\cdot,t)-\bar{w}(t)\|_{L^\infty} =0.
		\end{equation}
Next we define a number $w_*$ by
\begin{equation}\label{ws}	w_{*}=\frac{1}{|\Omega|}\|w_0\|_{L^1}-\frac{1}{|\Omega|}\int_0^\infty\int_\Omega uF(w).
\end{equation}
	Then integrating the third equation of \eqref{1-1} over $\Omega \times(0,t)$, we see that
		\begin{equation*}\label{vbar}
		\bar{w}(t)=w_*+\frac{1}{|\Omega|}\int_t^\infty\int_\Omega uF(w),
		\end{equation*}
		which implies
		\begin{equation}\label{cw-5}
		\begin{split}
		\|\bar{w}(t)-w_*\|_{L^\infty}\leq& \frac{C_F}{|\Omega|} \int_t^{\infty} \|u(\cdot,s)\|_{L^1}ds \to 0 \ \ \mathrm{as}\ \ \ t\to\infty.
		\end{split}
		\end{equation}
Then combining \eqref{cw-4} and \eqref{cw-5}, one has}
\begin{equation*}
\begin{split}
\|w(\cdot,t)-w_*\|_{L^\infty}
&\leq \|w(\cdot,t)-\bar{w}(t)\|_{L^\infty}+\|\bar{w}(t)-w_*\|_{L^\infty}\to 0,\ \mathrm{as}\ \ t\to\infty,
\end{split}
\end{equation*}
which yields \eqref{cw-1}.

Next, we shall show $w_*>0$. Noting $F(w)\in C^1([0,\infty))$ and $F(0)=0$ and using the boundedness of $u$ and $w$, we can find $\xi\in (0,w)$ and $\mathcal{K}>0$ such that
\begin{equation*}\label{wp}
\frac{uF(w)}{w}=\frac{F(w)-F(0)}{w}\cdot u=F'(\xi)u\leq {\color{black}\|F'(\xi)\|_{L^{\infty}}}\|u\|_{L^\infty}:=\mathcal{K}.
\end{equation*}
Let $\tilde{w}(x,t)$ be the solution of the following system
\begin{equation*}\label{wp-1}
\begin{cases}
	\tilde{w}_t-\Delta \tilde{w}=-\mathcal{K}\tilde{w},\ \ &x\in\Omega,t>0,\\
		\frac{\partial \tilde{w}}{\partial\nu}=0,\ \ &x\in\partial\Omega,t>0,\\
		\tilde{w}(x,0)=w_0(x),&x\in\Omega.
		\end{cases}
		\end{equation*}
		Clearly, $\tilde{w}(x,t)$ is a sub-solution of $w(x,t)$ by the comparison principle, and hence
		\begin{equation}\label{wp-2}
		w(x,t)\geq \tilde{w}(x,t).
		\end{equation}
		On the other hand, using \cite[Lemma 3.1]{HPW-M3AS-2013}, we can find a constant $\Gamma_0>0$ such that for all $t\geq 1$
		\begin{equation*}
		\tilde{w}(x,t)=e^{-\mathcal{K} t}e^{\Delta t} w_0\geq e^{-\mathcal{K} t}\Gamma_0\int_\Omega w_0,
		\end{equation*}
		which combined with \eqref{wp-2} gives
		\begin{equation}\label{wp-3}
		w(x,t)\geq e^{-\mathcal{K} t}\Gamma_0\int_\Omega w_0, \ \mathrm{for \ all}\ \ t\geq 1.
		\end{equation}
		Multiplying the third equation of \eqref{1-1} by $\frac{1}{w}$, and integrating by parts with respect to $x\in\Omega$, one has
		\begin{equation*}
		\frac{d}{dt}\int_\Omega \ln w(x,t)=\int_\Omega \frac{|\nabla w|^2}{w^2}-\int_\Omega \frac{F(w)}{w}u{\color{black}\geq-\int_{\Omega}\frac{F(w)}{w}u},
		\end{equation*}
	which thus gives
	\begin{equation}\label{wc-1}
		\int_\Omega \ln w(x,t)\geq \int_\Omega \ln w(x,1)-\int_1^t\int_\Omega \frac{F(w)}{w}u.
	\end{equation}
		Then using \eqref{wp-3} and the fact $\int_0^t\int_\Omega u \leq c_7$, from \eqref{wc-1} we can find a constant $c_8>0$ such that
		\begin{equation*}
		\int_\Omega \ln w(x,t)\geq -c_8, \ \mathrm{for\ all}\ t\geq 1,
		\end{equation*}
		which combined with the fact \eqref{cw-1} implies $w_*>0$.
	\end{proof}
In summary, we have the asymptotic behavior of solutions for the  system \eqref{1-1} with $\theta>0$.
\begin{proposition}\label{LT1}
Let the conditions of Theorem \ref{LTB} hold and $\theta>0$, the solution of  system \eqref{1-1} satisfies
\begin{equation*}
\lim\limits_{t\to\infty}(\|u(\cdot,t)\|_{L^\infty}+\|v(\cdot,t)\|_{L^\infty}+\|w(\cdot,t)-w_*\|_{L^\infty})=0,
\end{equation*}
where $w_{*}> 0$ defined by \eqref{ws}.
\end{proposition}
	
\subsection{Case of $\theta=0$}
	In this subsection, we shall study the large time behavior of the system \eqref{1-1} with $\theta=0$. We first show the decay of $w$ based on some ideas in \cite{Winkler-2014-ARMA}.
\begin{lemma}\label{Ac-1}
Assume the conditions in {Theorem \ref{LTB}} hold. Let $(u,v,w)$ be the solution of the system \eqref{1-1} with $\theta=0$. Then we have
\begin{equation}\label{Ac-2}
\int_0^\infty\int_\Omega uF(w)<\infty
\end{equation}
and
\begin{equation}\label{Ac-3}
\int_0^\infty \int_\Omega|\nabla w|^2<\infty.
\end{equation}
\end{lemma}
\begin{proof} Integrating the third equation of $\eqref{1-1}$ over $\Omega$ and using the homogeneous Neumann boundary condition, one has
\begin{equation*}
\frac{d}{dt}\int_\Omega w+\int_\Omega uF(w)=0,
\end{equation*}
which gives
\begin{equation}\label{Ac-4}
\int_0^t\int_\Omega uF(w)\leq \int_\Omega w_0,\ \ \ \ ~~\mathrm{for~~all}~~t>0,
\end{equation}
and $\eqref{Ac-2}$ is a direct result of $\eqref{Ac-4}$. We multiply the third equation of $\eqref{1-1}$ by $w$ to obtain
\begin{equation}\label{Ac-5}
\frac{1}{2}\frac{d}{dt}\int_\Omega w^2=-\int_\Omega|\nabla w|^2-\int_\Omega uwF(w).
\end{equation}
Integrating $\eqref{Ac-5}$ with respect to $t$ and using the nonnegativity of $u$ and $w$, one can derive
\begin{equation*}
\int_0^t\int_\Omega|\nabla w|^2\leq\frac{1}{2}\int_\Omega w_0^2,
\end{equation*}
which gives $\eqref{Ac-3}$.
\end{proof}

\begin{lemma}\label{Ac*} Let the conditions in Lemma \ref{Ac-1} hold. Then there exists a time sequence $(t_k)_{k\in\mathbb{N}}\subset(0,\infty)$ satisfying $t_k\to\infty$ as $k\to\infty$ such that
\begin{equation}\label{Ac-6}
\int_{t_k}^{t_k+1}\int_\Omega w\to 0~~~~\mathrm{as}~~~k\to \infty.
\end{equation}
\end{lemma}
\begin{proof}From $\eqref{Ac-2}$, we have
\begin{equation}\label{Ac-7}
\int_j^{j+1}\int_\Omega uF(w)\to0~~~\mathrm{as}~~j\to\infty.
\end{equation}
Defining $\bar{F}(w):=\frac{1}{|\Omega|}\int_\Omega F(w)$, we have
\begin{equation}\label{Ac-8}
\begin{split}
\int_j^{j+1}\int_\Omega uF(w)
=\int_j^{j+1}\int_\Omega u(F(w)-\bar{F}(w))+\int_j^{j+1}\int_\Omega u\bar{F}(w)
=I_1(j)+I_2(j).
\end{split}
\end{equation}
Using the H\"{o}lder inequality, Poinca\'{r}e inequality and the boundedness of $\|u(\cdot,t)\|_{L^2}$ in $\eqref{L2}$, one has
\begin{equation}\label{Ac-9}
\begin{split}
\big|I_1(j)\big|
&\leq \left(\int_j^{j+1}\int_\Omega u^2\right)^\frac{1}{2}\cdot \left(\int_j^{j+1}\int_\Omega |F(w)-\bar{F}(w)|^2\right)^\frac{1}{2}\\
&\leq c_1 \left(\int_j^{j+1}\int_\Omega |F(w)-\bar{F}(w)|^2\right)^\frac{1}{2}\\
&\leq c_1\left(c_2\int_j^{j+1}\int_\Omega|\nabla F(w)|^2\right)^\frac{1}{2}\\
&\leq c_3 \left(\int_j^{j+1}\int_\Omega|\nabla w|^2\right)^\frac{1}{2}\to 0,~~~\mathrm{as}~~j\to\infty.
\end{split}
\end{equation}
where we have used $\eqref{Ac-3}$ to derive the convergence. Then combining $\eqref{Ac-7},\eqref{Ac-8}$ and $\eqref{Ac-9}$, one has $I_2(j)\to 0$ as $j\to\infty$.

On the other hand, using $\eqref{L1u}$ and the fact $\|w(\cdot,t)\|_{L^1}\leq \|w_0\|_{L^1}$, we have
 \begin{equation*}\label{Ac-10}
 \|u_0\|_{L^1}+\alpha\|w_0\|_{L^1}=\|u\|_{L^1}+\alpha\|w\|_{L^1}\leq \|u\|_{L^1}+\alpha\|w_0\|_{L^1},
 \end{equation*}
 which implies
 \begin{equation}\label{Ac-11}
 \|u_0\|_{L^1}\leq\|u\|_{L^1} \leq\|u_0\|_{L^1}+\alpha{\color{black}\|w_0\|_{L^1}}.
 \end{equation}
 Hence, using $\eqref{Ac-11}$ and the fact $I_2(j)\to 0$ as $j\to\infty$, one has
 \begin{equation*}
\bar{u}_0\int_j^{j+1}\int_\Omega F(w) =\|u_0\|_{L^1}\int_j^{j+1}\bar{F}(w)\leq I_2(j)=\int_j^{j+1}\int_\Omega u\bar{F}(w) \to 0,~~\mathrm{as}~~j\to\infty,
 \end{equation*}
 which implies
 \begin{equation}\label{Ac-12}
 \int_j^{j+1}\int_\Omega F(w)\to0,\ \ \ \ ~~\mathrm{as}~~j\to\infty.
 \end{equation}
Next, we will show that $\eqref{Ac-12}$ implies $\eqref{Ac-6}$. In fact, if we define $w_j(x,s):=w(x,j+s),(x,s)\in\Omega\times(0,1),j\in\mathbb{N}$, then $\eqref{Ac-12}$ implies
\begin{equation*}
\int_0^1\int_\Omega F(w_j(x,s))\to 0, \ \ \ \ \ ~~~\mathrm{as}~~j\to\infty.
\end{equation*}
Hence we can extract a subsequence $(j_k)_{k\in\mathbb{N}}\subset\mathbb{N}$ such that $j_k\to\infty$ and $F(w_{j_{k}})\to 0$ almost everywhere in $\Omega\times(0,1)$ as $k\to\infty$. Because the function $F$ is positive on $(0,\infty)${\color{black}~and~$F(0)=0$}, which requires that $w_{j_{k}}\to 0$ almost everywhere in $\Omega\times(0,1)$ as $k\to\infty$. Moreover, since $\|w(\cdot,t)\|_{L^\infty}\leq c_4$ for all $t>0$, then the sequence $(w_{j_{k}})_{k\in\mathbb{N}}\to 0$ in $L^1(\Omega\times (0,1))$ as $k\to\infty$. Choosing $t_k:=j_k$, one has $\eqref{Ac-6}$. Then the proof of this lemma is completed.
\end{proof}
\begin{lemma}\label{Lw}Suppose the conditions in {\color{black}Lemma \ref{Ac-1}} hold. Let $(u,v,w)$ be the solution of the system \eqref{1-1} with $\theta=0$. Then it holds that
		\begin{equation}\label{Lw-1}
		\|w(\cdot,t)\|_{L^\infty}\to 0~~~~\mathrm{as}~~~t\to\infty.
		\end{equation}
	\end{lemma}
\begin{proof}Letting $(t_k)_{k\in\mathbb{N}}\subset(0,\infty)$ be the sequence chosen in Lemma $\ref{Ac*}$. Using the Gagliardo-Nirenberg inequality, one can find a constant $c_1>0$ such that
		\begin{equation}\label{Lw-6}
		\begin{split}
		\|w(\cdot,t)\|_{L^\infty}
		&\leq c_1\|\nabla w(\cdot,t)\|_{L^4}^\frac{4}{5}\|w(\cdot,t)\|_{L^1}^\frac{1}{5}+c_1\|w(\cdot,t)\|_{L^1}\\
		&\leq \mu \|\nabla w(\cdot,t)\|_{L^4}+c_2\|w(\cdot,t)\|_{L^1},
		\end{split}
		\end{equation}
		where $\mu>0$ is an arbitrary constant, $c_2>0$ is a constant depending on $\mu$. Noting the uniform boundedness of $\|\nabla w(\cdot,t)\|_{L^4}$ and the arbitrary of $\mu$, and using $\eqref{Ac-6}$, from \eqref{Lw-6} we get
		\begin{equation*}\label{A3-1}
		\int_{t_k}^{t_k+1}\|w(\cdot,t)\|_{L^\infty}\to 0,\ \ \ ~~~\mathrm{as}~~k\to\infty,
		\end{equation*}
		which implies
		\begin{equation}\label{Lw-7}
		\liminf_{t\to\infty}\|w(\cdot,t)\|_{L^\infty}=0.
		\end{equation}
		The combination of \eqref{Lw-7} and the fact that $t\to\|w(\cdot,t)\|_{L^\infty}$ is monotone as shown in Lemma \ref{UW}, one obtains \eqref{Lw-1} and completes the proof of Lemma \ref{Lw}.
\end{proof}

	\begin{lemma}\label{Lu}
		Let $(u,v,w)$ be the solution of the system \eqref{1-1} with $\theta=0$. Then it follows that
		\begin{equation}\label{Lu-1}
		\frac{d}{dt}\int_\Omega (u-u_*)^2+\gamma_1\int_\Omega|\nabla u|^2\leq \int_\Omega \frac{|\gamma'(v)|^2}{\gamma(v)}u^2 |\nabla v|^2+\alpha M^2 \int_\Omega F(w),
		\end{equation}
		where $M>0$ is defined by \eqref{M} and  $u_{*}=\bar{u}_{0}+\alpha\bar{w}_{0}$.
	\end{lemma}
\begin{proof}
	We rewrite the first equation of the system \eqref{1-1} as
	\begin{equation}\label{Lu-2}
		(u-u_*)_t=\Delta (\gamma(v)u)+\alpha uF(w).
	\end{equation}
	Then multiplying \eqref{Lu-2} by $u-u_*$, and integrating it by parts, we end up with
	\begin{equation*}\label{Lu-3}
		\begin{split}
		&\frac{1}{2}\frac{d}{dt}\int_\Omega (u-u_*)^2+\int_\Omega \gamma(v)|\nabla u|^2\\
		&=-\int_\Omega \gamma'(v)u\nabla u\cdot\nabla v+\alpha \int_\Omega uF(w)(u-u_*)\\
		&\leq \frac{1}{2}\int_\Omega \gamma(v)|\nabla u|^2+\frac{1}{2}\int_\Omega\frac{|\gamma'(v)|^2}{\gamma(v)}u^2|\nabla v|^2+\alpha \int_\Omega u^2F(w),
		\end{split}
		\end{equation*}
		which, together with the facts $\|u\|_{L^\infty}\leq M $ and $\gamma(v)\geq \gamma_1$, gives
		\begin{equation*}
		\begin{split}
		\frac{d}{dt}\int_\Omega (u-u_*)^2+\gamma_1\int_\Omega |\nabla u|^2
		&\leq \int_\Omega\frac{|\gamma'(v)|^2}{\gamma(v)}u^2|\nabla v|^2+\alpha M^2 \int_\Omega F(w),\\
		\end{split}
		\end{equation*}
		and hence \eqref{Lu-1} follows.
	\end{proof}
	\begin{lemma}\label{LV}
		The solution $(u,v,w)$ of the system \eqref{1-1} with $\theta=0$ satisfies
		\begin{equation}\label{LV-1}
		\frac{d}{dt}\int_\Omega(v-u_*)^2+2D\int_\Omega|\nabla v|^2+\int_\Omega(v-u_*)^2\leq \int_\Omega (u-u_*)^2.
		\end{equation}
	\end{lemma}
	\begin{proof}
		Multiplying the second equation of the system \eqref{1-1} by $v-u_*$, and integrating by parts, we end up with
		\begin{equation*}
		\begin{split}
		\frac{1}{2}\frac{d}{dt}\int_\Omega (v-u_*)^2
		&=\int_\Omega (v-u_*)[D\Delta v+(u-u_*)-(v-u_*)]\\
		&=-D\int_\Omega|\nabla v|^2-\int_\Omega(v-u_*)^2+\int_\Omega (v-u_*)(u-u_*)\\
		&\leq -D\int_\Omega|\nabla v|^2-\frac{1}{2} \int_\Omega(v-u_*)^2+\frac{1}{2}\int_\Omega(u-u_*)^2,
		\end{split}
		\end{equation*}
		which yields \eqref{LV-1}.

	\end{proof}
	\begin{lemma}\label{Luv} Let $(u,v,w)$ be the solution of the system \eqref{1-1} with $\theta=0$. Then
		there exists a positive constant $D_1$ such that if $D\geq D_1$, it holds that
		\begin{equation}\label{Luv-1}
		\lim\limits_{t\to\infty}(\|u(\cdot,t)-u_*\|_{L^\infty}+\|v(\cdot,t)-u_*\|_{L^\infty})= 0.
		\end{equation}
	\end{lemma}
	\begin{proof}
		Applying the Poincar\'{e} inequality, we find a constant $C_p>0$ such that
		\begin{equation}\label{Luv-2}
		\int_\Omega (u-\bar{u})^2\leq C_p\int_\Omega |\nabla u|^2.
	\end{equation}
	Then using the definition of $u_{*}$, we know from \eqref{L1u} that $u_{*}=\bar{u}+\alpha\bar{w}$. Then it follows from \eqref{Luv-2} that
	\begin{equation*}
		\begin{split}
		\int_\Omega (u-u_*)^2
		&\leq 2\int_\Omega (u-\bar{u})^2+2{\color{black}\alpha^{2}} \int_\Omega \bar{w}^2\\
		&\leq 2C_p\int_\Omega |\nabla u|^2+\frac{2{\color{black}\alpha^{2}}}{|\Omega|}\left(\int_\Omega w\right)^2,
		\end{split}
		\end{equation*}
which implies
\begin{equation}\label{Luv-3}
		\begin{split}
		\frac{\gamma_1}{2C_p}\int_\Omega (u-u_*)^2
		&\leq \gamma_1\int_\Omega |\nabla u|^2+\frac{{\color{black}\alpha^{2}}\gamma_1}{C_p|\Omega|}\left(\int_\Omega w\right)^2.
		\end{split}
		\end{equation}
Applying \eqref{Luv-3} into \eqref{Lu-1}, and using the facts $\|u(\cdot,t)\|_{L^\infty}\leq M$, $0<\gamma_1\leq\gamma(v)$ and $|\gamma'(v)|\leq \eta$, we can derive that
		\begin{equation}\label{Luv-4}
\begin{split}
		&\frac{d}{dt}\int_\Omega (u-u_*)^2+\frac{\gamma_1}{2C_p}\int_\Omega (u-u_*)^2\\
&\leq \int_\Omega \frac{|\gamma'(v)|^2}{\gamma(v)}u^2 |\nabla v|^2+\alpha M^2 \int_\Omega F(w)+\frac{{\color{black}\alpha^{2}}\gamma_1}{C_p|\Omega|}\left(\int_\Omega w\right)^2\\
&\leq \frac{\eta^2M^2}{\gamma_1} \int_\Omega |\nabla v|^2+\alpha M^2 \int_\Omega F(w)+\frac{{\color{black}\alpha^{2}}\gamma_1}{C_p|\Omega|}\left(\int_\Omega w\right)^2.
\end{split}
		\end{equation}
		On the other hand, we multiply \eqref{LV-1} by $\frac{\gamma_1}{4c_P}$, and use \eqref{Luv-4} to have
		\begin{equation}\label{Luv-5}
		\begin{split}
		&\frac{d}{dt}\left(\int_\Omega (u-u_*)^2+\frac{\gamma_1}{4C_p}\int_\Omega (v-u_*)^2\right)+\frac{\gamma_1}{4C_p}\int_\Omega (u-u_*)^2+\frac{\gamma_1}{4C_p}\int_\Omega (v-u_*)^2\\
		&\leq \left( \frac{\eta^2M^2}{\gamma_1}-\frac{D\gamma_1}{2C_p}\right)\int_\Omega|\nabla v|^2+\alpha M^2 \int_\Omega F(w)+\frac{{\color{black}\alpha^{2}}\gamma_1}{C_p|\Omega|}\left(\int_\Omega w\right)^2.\\
		\end{split}
		\end{equation}
Using the definition of $M$ in \eqref{M}, one can find two constants $C_1,C_2>0$ independent of $D$ such that
\begin{equation*}
M:=C_1{\color{black}(1+\alpha)^{13}}\left(1+\frac{1}{D}\right)^{12}e^{C_2(1+\frac{1}{D})^{4}(1+\alpha)^{6}}.
\end{equation*}
Let $D_*$ be the positive constant uniquely determined by the following identity
\begin{equation*}
D_*=\frac{2\eta^2C_p{\color{black}C_{1}^{2}}}{\gamma_1^2}
{\color{black}(1+\alpha)^{26}\left(1+\frac{1}{D_*}\right)^{24}e^{2C_2(1+\frac{1}{D_*})^{4}(1+\alpha)^{6}}},
\end{equation*}
where $C_p,\gamma_1$ and $\eta$ are independent of $D_*$. Then if $D\geq  D_*$, one has $\frac{\eta^2M^2}{\gamma_1}-\frac{D\gamma_1}{2C_p}\leq 0$, and hence the estimate \eqref{Luv-5} becomes
\begin{equation}\label{Luv-6}
		\begin{split}
		&\frac{d}{dt}\left(\int_\Omega (u-u_*)^2+\frac{\gamma_1}{4C_p}\int_\Omega (v-u_*)^2\right)+\frac{\gamma_1}{4C_p}\int_\Omega (u-u_*)^2+\frac{\gamma_1}{4C_p}\int_\Omega (v-u_*)^2\\
		&\leq \alpha M^2 \int_\Omega F(w)+\frac{{\color{black}\alpha^{2}}\gamma_1}{C_p|\Omega|}\left(\int_\Omega w\right)^2.\\
		\end{split}
		\end{equation}
	Define $Z(t):=\int_\Omega (u-u_*)^2+\frac{\gamma_1}{4C_p}\int_\Omega (v-u_*)^2$ and $G(t):=\alpha M^2 \int_\Omega F(w)+\frac{{\color{black}\alpha^{2}}\gamma_1}{C_p|\Omega|}\left(\int_\Omega w\right)^2$. Choosing $c_1:=\min\{1,\frac{\gamma_1}{4C_p}\}$, we have from \eqref{Luv-6}
		\begin{equation}\label{Luv-7}
		Z'(t)+c_1Z(t)\leq G(t).
		\end{equation}
Since $\|w(\cdot,t)\|_{L^\infty}\to 0$ as $t\to\infty$ (see Lemma \ref{Lw}), one has $G(t)\to 0$ as $t\to\infty$. Then from \eqref{Luv-7}, we can derive that
	\begin{equation*}
		Z(t)\to 0 \ \ \mathrm{as} \ \ t\to \infty.
	\end{equation*}
	This implies
	\begin{equation}\label{L2-u}
		\|u(\cdot,t)-u_*\|_{L^2}+\|v(\cdot,t)-u_*\|_{L^2}\to 0\ \ \mathrm{as} \ \ t\to \infty.
	\end{equation}
Using the similar arguments as in Lemma \ref{cu} with \eqref{L2-u}, we obtain \eqref{Luv-1} directly.
Then the proof of Lemma \ref{Luv} is completed.
\end{proof}
{\color{black}In summary, we have the following asymptotic results for the case $\theta=0$.}
\begin{proposition}\label{LT0}
	Suppose the conditions in Theorem \ref{GB} hold. Let $(u,v,w)$ be the solution of the system \eqref{1-1} with $\theta=0$. Then there exists constant $D_1>0$ such that if $D\geq D_1$, it has
\begin{equation*}
\lim\limits_{t\to\infty}(\|u(\cdot,t)-u_*\|_{L^\infty}+\|v(\cdot,t)-u_*\|_{L^\infty}+\|w(\cdot,t)\|_{L^\infty})=0,
\end{equation*}
where $u_*=\frac{1}{|\Omega|}(\|u_0\|_{L^1}+\alpha\|w_0\|_{L^1})$.
\end{proposition}

\begin{proof}[Proof of Theorem \ref{LTB}] The proof of Theorem \ref{LTB} is a direct consequence of Proposition \ref{LT1} and Proposition \ref{LT0}.
\end{proof}

\section{Simulations and discussions}
\subsection{Linear instability analysis}
The results of Theorem \ref{LTB} imply that the system \eqref{1-1} has no pattern formation if $\theta>0$ or $\theta=0$ and $D$ is large. In this section, we will study the possible pattern arising from the system \eqref{1-1} with $\theta=0$ and small $D$.	To this end, we  note that \eqref{1-1} with $\theta=0$ has three constant equilibria $(0,0,0)$, $(0,0,\frac{u_*}{\alpha})$ and $(u_*, u_*, 0)$ for given initial value $(u_0,v_0,w_0)$, where
$
	u_*=\bar{u}_0+\alpha\bar{w}_0.
$
We first consider the system \eqref{1-1} with $\theta=0$ in the absence of spatial components, that is
	\begin{equation*}\label{ODE}
	\begin{cases}
	u_{t}=\alpha uF(w),\\
	v_{t}=u-v,\\
	w_{t}=-uF(w).
	\end{cases}
	\end{equation*}
The linear stability/instability of each equilibrium is determined by the sign of the eigenvalues $(\rho_{1},\rho_{2},\rho_{3})$ defined by
\begin{equation*}
(\rho_{1},\rho_{2},\rho_{3})=
\begin{cases}
	(0,-1,0),&\text{at}~~(0,0,0),\\
	(0,-1,\alpha F(\frac{u_*}{\alpha})),&\text{at}~~(0,0,\frac{u_*}{\alpha}),\\
	(0,-1,-u_{*}F'(0)),&\text{at}~~(u_{*},u_{*},0).
\end{cases}
\end{equation*}
Since $F(\frac{u_*}{\alpha})>0$ and $F'(0)>0$, we know  the non-trivial steady state $(0,0,\frac{u_*}{\alpha})$ is linearly unstable, while $(0,0,0)$ and $(u_*,u_*,0)$ are linearly stable.
Hence we study the possible patterns bifurcating from the constant equilibria $(u_c, u_c, 0)$ where $u_c=0$ or $u_c=u_*$. To this end, we linearize the system \eqref{1-1} at the equilibrium $(u_c,u_c,0)$ to obtain
	\begin{equation}\label{PF-1}
	\begin{cases}
	\Phi_t=A_1\Delta \Phi+B_1\Phi, &x\in\Omega, \ t>0,\\
	(\nu\cdot \nabla)\Phi=0,&x\in\partial\Omega, \ t>0,\\
	\Phi(x,0)=(u_0-u_c,v_0-u_c,w_0)^{\mathcal{T}},&x\in \Omega,
	\end{cases}
	\end{equation}
	where $\mathcal{T}$ denotes the transpose and
	\begin{equation*}
	\Phi=\left(
	\begin{array}{c}
	u-u_c\\
	v-u_c\\
	w\\
	\end{array}
	\right),
	\ \ \ \ A_1=\left(
	\begin{array}{ccc}
	\gamma(u_c) & \gamma'(u_c)u_c &0 \ \\
	0 & D & 0 \\
	0&0&1
	\end{array}
	\right)
	\end{equation*}
	as well as
	\begin{equation*}
	B_1=\left(
	\begin{array}{ccc}
	0 & 0 &\alpha u_c F'(0)\\
	1 & -1 &0 \\
	0 &0 &-u_c F'(0)
	\end{array}
	\right)
	\end{equation*}
	Noting that the linear system \eqref{PF-1} has the solution of the form
	\begin{equation}\label{PF-2}
	\Phi(x,t)=\sum\limits_{k\geq 0} c_k e^{\rho t} W_k(x),
	\end{equation}
	where $W_k(x)$ denotes the eigenfunction of the following eigenvalue problem:
	\begin{equation*}\label{hn}
	\Delta W_k(x)+k^2 W_k(x)=0, \ \ \frac{\partial W_k(x)}{\partial \nu}=0,
	\end{equation*}
	and the constants $c_k$ are determined by the Fourier expansion of the initial conditions in terms of $W_k(x)$ and $\rho$ is the temporal eigenvalue. After some calculations, we know $\rho$ is the eigenvalue of the following matrix
	\begin{equation*}
	\begin{split}
	M_k
	&=\left(
	\begin{array}{ccc}
	- \gamma(u_c)k^2 & -u_c\gamma'(u_c)k^2& \alpha u_cF'(0)\ \\
	1&-Dk^{2}-1& 0\\
	0 & 0&-k^{2}-u_cF'(0)\\
	\end{array}
	\right).\\
	\end{split}
	\end{equation*}
	Obviously, $\rho(k^2)=-k^{2}-u_uF'(0)$ is an eigenvalue, which is negative for all $k\ne 0$. Hence to get the possible pattern formation, we only need to consider the other two eigenvalues of the matrix $M_k$, which satisfy
	\begin{equation*}\label{LP}
	\rho^{2}+a_{1}(k^2)\rho+a_{0}(k^2)=0,
	\end{equation*}
	where
	\begin{equation*}\label{a2}
	\begin{cases}
	a_{1}(k^2)&=1+(D+\gamma(u_c))k^{2}>0,\\[2mm]
	a_{0}(k^2)&=D\gamma(u_c)k^{4}+\left(\gamma(u_c)+u_c\gamma'(u_c)\right)k^{2}.
	\end{cases}
	\end{equation*}
	One can check that if $\gamma(u_c)+u_c\gamma'(u_c)\geq 0$ which is the case for $u_c=0$, then $a_{0}(k^2)>0$for all $k\ne 0$,  which implies the real part of the eigenvalues $\rho(k^2)$ are negative, and hence the steady state $(0,0,0)$ is linearly stable and no patterns will bifurcate from $(0,0,0)$. Next we consider the equilibrium $(u_*, u_*, 0)$. If $\gamma(u_{*})+u_{*}\gamma'(u_{*})< 0$, the real part of the eigenvalues $\rho(k^2)$ can be positive and hence the pattern formation may occur provided that the admissible wavenumber $k$ satisfies
	\begin{equation}\label{WN}
	0<k^{2}<-\frac{\gamma(u_{*})+u_{*}\gamma'(u_{*})}{D\gamma(u_{*})}=:\bar{k}.
	\end{equation}
Note the allowable wave numbers $k$ are discrete in a bounded domain, for instance if $\Omega=(0,l)$ then $k=\frac{n\pi}{l}$ for $n=1,2, \cdots$. Hence the condition (\ref{WN}) is only necessary because the interval $(0,\bar{k})$ may not contain any desired discrete number $k^2$, for instance when $D>0$ is sufficiently large.
Hence we have the following conclusion.
	\begin{lemma}
		Suppose $\gamma(v)$ satisfies the assumptions (H1). Then the homogeneous steady state $(u_*,u_*,0)$ of the system \eqref{1-1} is linearly unstable if and only if $\gamma(u_{*})+u_{*}\gamma'(u_{*})< 0$ and there is at least an allowable wavenumber $k$ satisfying condition \eqref{WN}.
\end{lemma}

\subsection{Simulations and questions}
 In section 5.1, we identify the instability parameter regimes for the possible pattern formation. But this linear instability result is not sufficient to conclude that there are non-constant stationary (pattern) solutions. Now we want to numerically test in one dimension whether non-constant stationary patterns exist for $\gamma(v)$ satisfying the conditions in Lemma \ref{WN}. For definiteness in the simulation, we assume $\Omega=(0,l)$ and consider
	\begin{equation*}
	\gamma(v)=\gamma_1+\gamma_0 e^{-\lambda v}, \ \ F(w)=\frac{w^2}{1+w^2}
	\end{equation*}
where $\gamma_0, \gamma_1$ and $\lambda$ are positive constants. Then the condition $\gamma(u_{*})+u_{*}\gamma'(u_{*})=\gamma_1+\gamma_0 e^{-\lambda u_*}(1-\lambda u_*)< 0$ in Lemma \ref{WN} amounts to
\begin{equation}\label{con1}
u_*>\frac{1}{\lambda} \  \ \text{and}\ \ \frac{\gamma_1}{\gamma_0}e^{\lambda u_*}<\lambda u_*-1.
\end{equation}
Since $k=\frac{\pi}{l}$, then the condition \eqref{WN} becomes
\begin{equation}\label{con2}
D<-\frac{\gamma(u_{*})+u_{*}\gamma'(u_{*})}{\gamma(u_{*})}\cdot \frac{l^2}{(n\pi)^2}=\frac{\gamma_1+\gamma_0e^{-\lambda u_*}(1-\lambda u_*)}{\gamma_1+\gamma_0e^{-\lambda u_*}}\cdot \frac{l^2}{(n\pi)^2}.
\end{equation}
Therefore if we choose appropriate values of $\gamma_0, \gamma_1, \lambda, u_*$ and $l$ so that the conditions \eqref{con1}-\eqref{con2} hold for some positive integer $n$, the pattern formation is expected from the results of Lemma \ref{WN}. Note that $u_*=\bar{u}_0+\alpha\bar{w}_0$. Hence for numerical simulations, we choose the initial value $(u_0, v_0, w_0)$ as a small random perturbation of the equilibrium $(u_*,v_*,0)$, and fix $\lambda=\alpha=1$. The system \eqref{1-1} is numerically solved by the MATLAB PDEPE solver. We choose $l=20, (u_*,v_*,0)=(4,4,0), \gamma_0=10, \gamma_1=0.1$ and show the numerical simulations for $D=0.1$ and $D=0.01$ in Fig.\ref{fig1} where we do observe the aggregated stationary patterns. This indicates for suitably small $D>0$, the system \eqref{1-1} with appropriate motility function $\gamma(v)$ admits the pattern formation, which complements the analytical results of Theorem \ref{LTB}. However the rigorous proof the existence of pattern (stationary) solutions leaves open in this paper and we shall investigate this question in the future. Note that the assumption (H1) rules out the possible degeneracy of  motility function $\gamma(v)$, which plays a key role in proving the results of this paper. Therefore another interesting open question is the global dynamics of \eqref{1-1} without assuming that $\gamma(v)$ has a positive lower bound such as $\gamma(v)=(1+v)^{-\lambda}$ or $\gamma(v)=e^{-\lambda v}$ with $\lambda>0$. Such motility function $\gamma(v)$ without positive lower bound has been used to study the global boundedness/asymptotics of solutions and stationary solutions for the two-component density-suppressed motility model (\ref{OR-2}) in \cite{JKW-SIAP-2018, MPW-PD-2019}, where the quadratic decay $-\mu u^2$ plays an essential role. However the three-component system \eqref{1-1} does not have such nice decay term and hence novel ideas are anticipated to solve the above-mentioned open question.

\begin{figure}[htbp]
\centering
\includegraphics[width=7cm,angle=0]{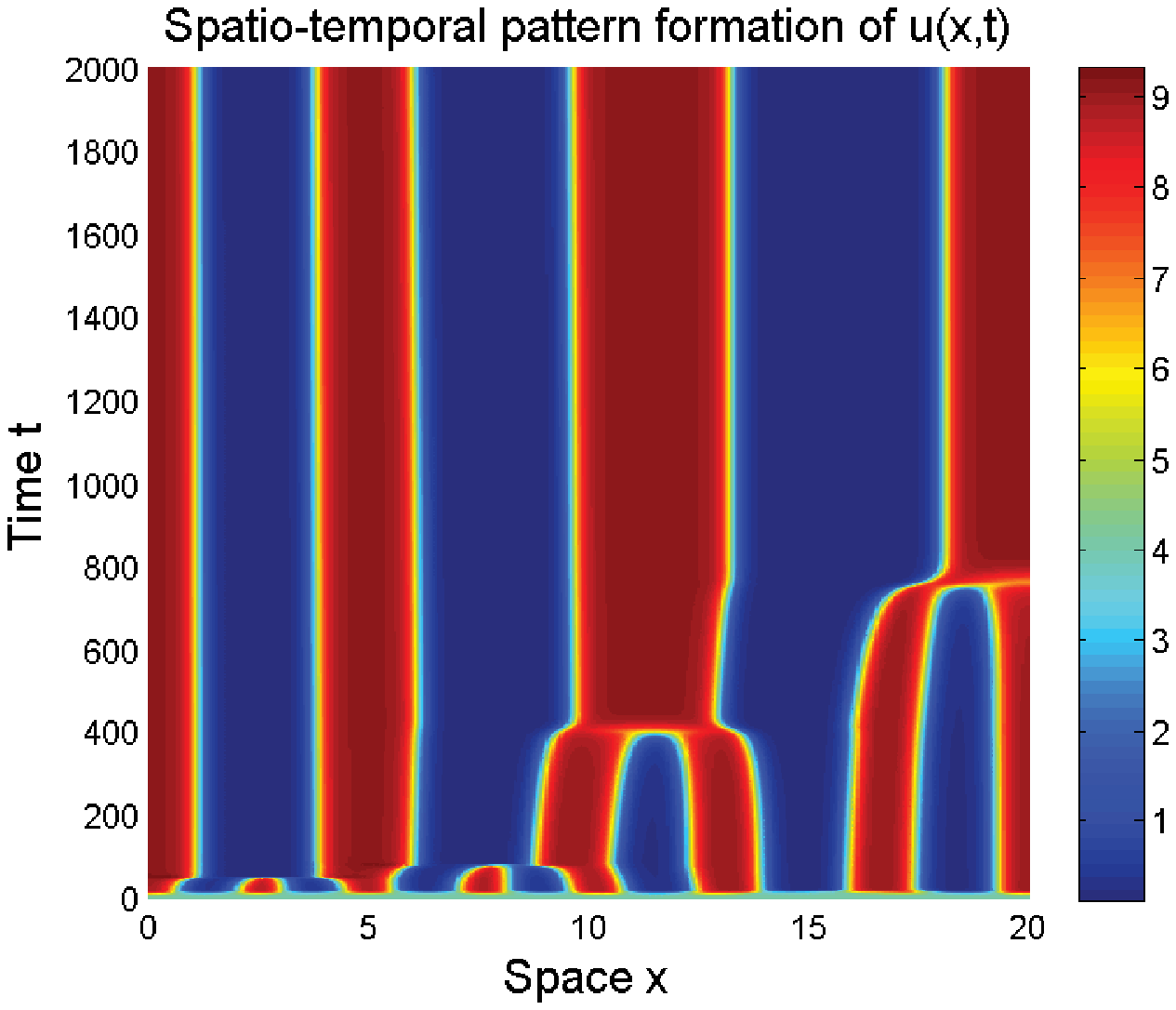}
\includegraphics[width=7cm,angle=0]{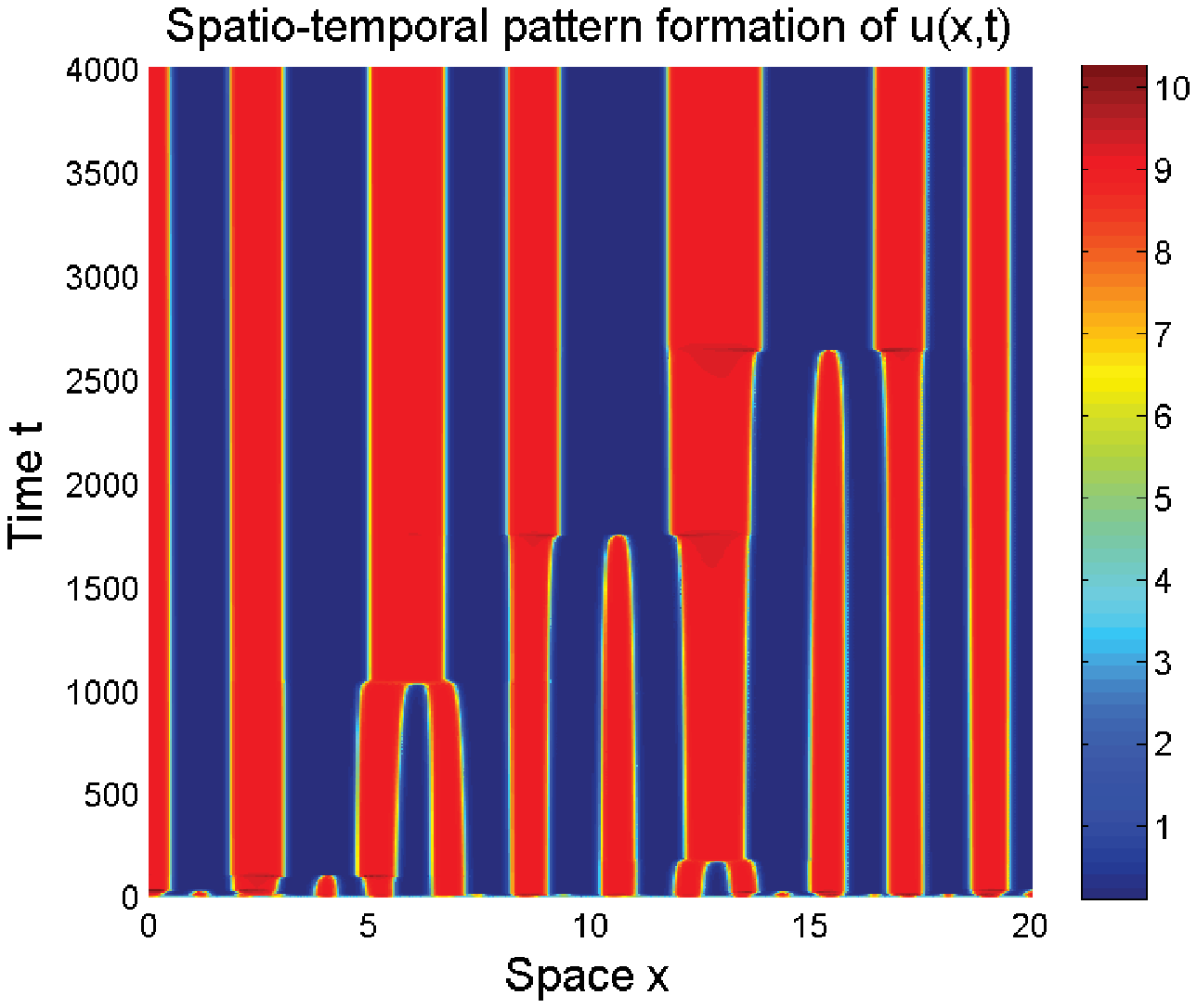}

\includegraphics[width=7cm,angle=0]{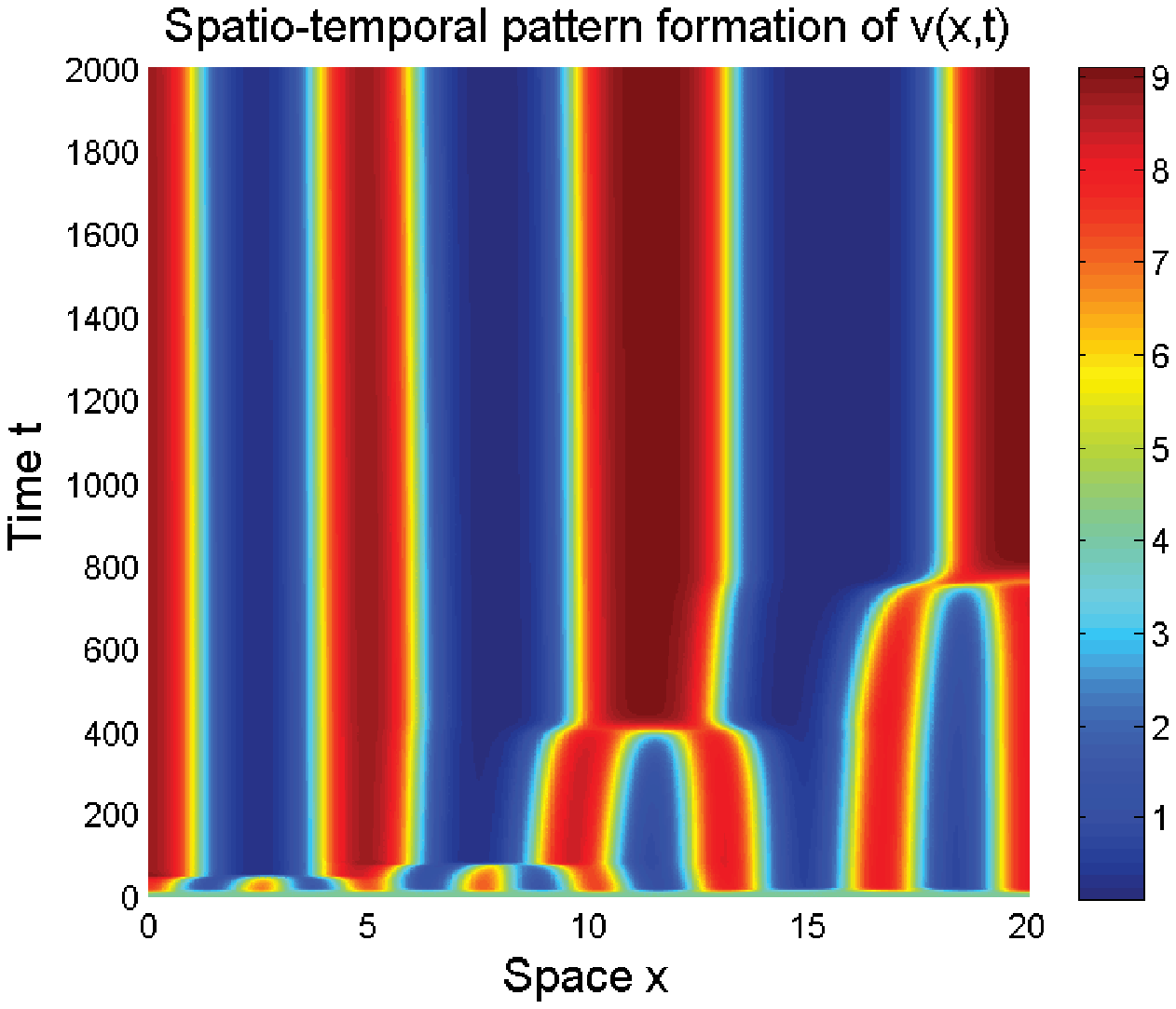}
\includegraphics[width=7cm,angle=0]{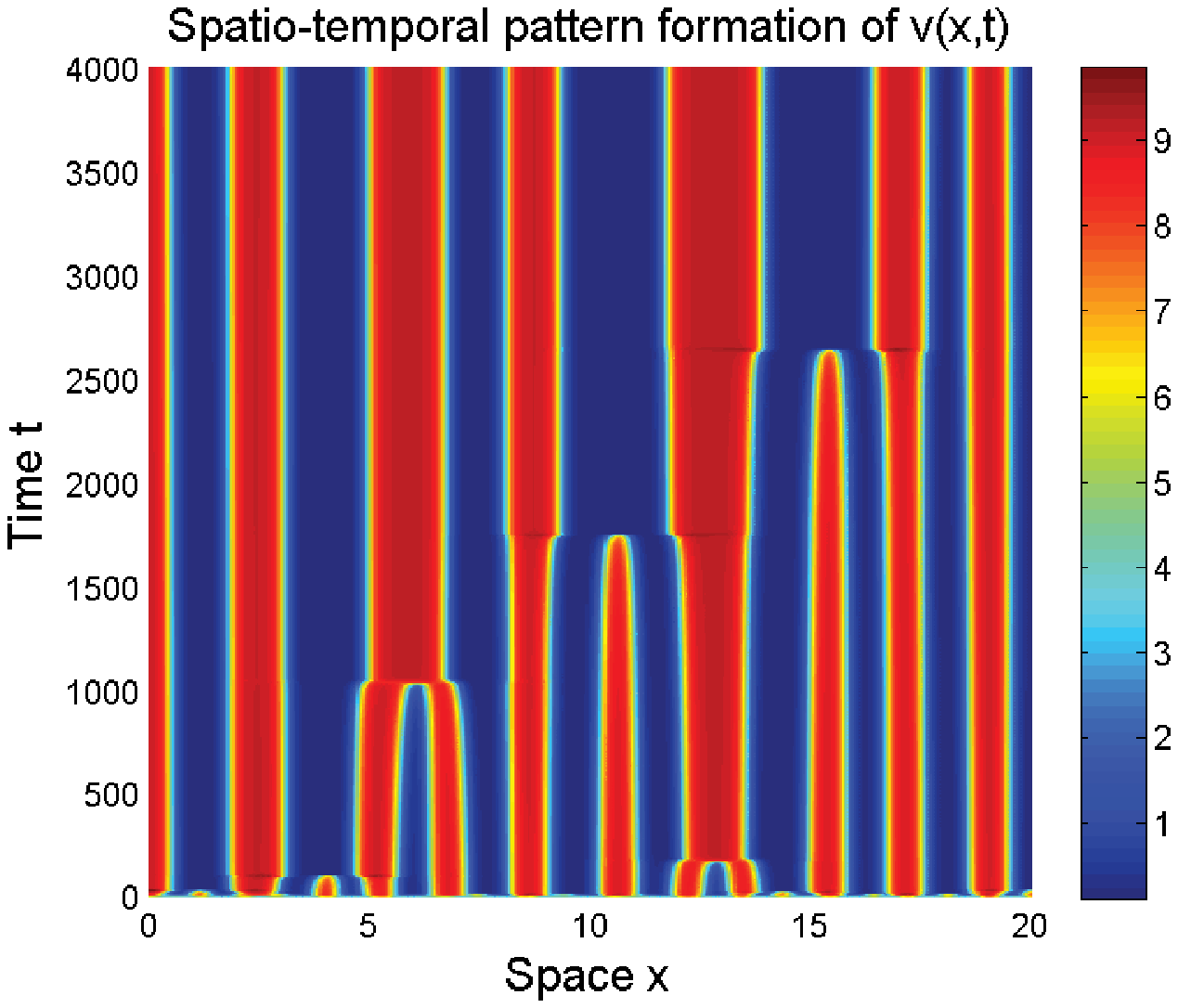}

$D=0.1$ \hspace{5cm} $D=0.01$
\caption{Numerical simulations of the pattern formation to the system \eqref{1-1} in $[0,20]$ with $\gamma(v)=0.1+10e^{-v}, F(w)=\frac{w^2}{1+w^2}$, where $\lambda=\alpha=1$ and $(u_0, v_0, w_0)$ is set as a small random perturbation of the constant steady state $(4,4,0)$. There is no pattern formation for $w$ and hence the numerical simulation of $w$ is not shown here.}
\label{fig1}
\end{figure}
	
%
%

\bigbreak
\noindent \textbf{Acknowledgment.}
We are grateful to the referee for several helpful comments improving our results. The research of H.Y. Jin was supported by the NSF of China (No. 11871226), Guangdong Basic and Applied Basic Research Foundation (No. 2020A1515010140), Guangzhou Science and Technology Program (No. 202002030363) and the Fundamental Research Funds for the Central Universities. The research of S. Shi was supported by Project Funded by the NSF of China No. 11901400. The research of Z.A. Wang was supported by the Hong Kong RGC GRF grant No. 15303019 (Project P0030816).

\end{document}